\documentclass[a4paper]{article}

\usepackage[english]{babel}
\usepackage{amsmath}
\usepackage{graphicx}
\usepackage[colorinlistoftodos]{todonotes}
\usepackage[all]{xy}
\usepackage{dsfont}
\usepackage{amssymb}
\usepackage{bbold}
\usepackage{amsthm}

\usepackage{tikz}
\usepackage{verbatim}
\usetikzlibrary{matrix}
\usetikzlibrary{cd}

\usepackage[backend=bibtex,style=alphabetic,natbib=true]{biblatex} 
\usepackage[autostyle=true]{csquotes}

\usepackage{pdfpages}

\usepackage{hyperref}

\newtheorem{theorem}{Theorem}[section]
\newtheorem{proposition}{Proposition}[theorem]
\newtheorem{lemma}{Lemma}[theorem]
\newtheorem{corollary}{Corollary}[theorem]
\newtheorem*{definition}{Definition}

\newcommand{\Set}{\textbf{Set}}

\newcommand{\Sub}{\texttt{Sub}}

\newcommand{\C}{\textbf{C}}

\newcommand{\T}{\textsl{T}}
\newcommand{\op}{^{op}}

\newcommand{\D}{\Delta}

\newcommand{\Dn}{\Delta^n}
\newcommand{\Dm}{\Delta^m}

\newcommand{\too}[1]{\overset{#1}{\longrightarrow}}
\newcommand{\hookar}[1]{\overset{#1}{\hookrightarrow}}
\newcommand{\colim}[1]{ \underset{#1}{\texttt{colim}} \text{ } }
\newcommand{\fset}[1]{ \{ #1_0 , \dots , #1_n \} }

\newcommand{\Ima}{\texttt{Im}}
\newcommand{\ims}{\Ima(\sigma)}
\newcommand{\imt}{\Ima(\tau)}

\newcommand{\G}{\mathcal{G}}

\newcommand{\Sh}{\texttt{Sh}}
\newcommand{\IpC}{IC_\cdot^{\bar{p}}}

\newcommand{\s}{\sigma}
\newcommand{\lst}{\texttt{lst}}
\newcommand{\st}{\texttt{st}}
\newcommand{\lsts}{\lst(\s)}
\newcommand{\fst}{\texttt{fst}}
\newcommand{\mst}[1]{\texttt{mst}_{#1}}

\title{A construction of intersection cohomology from a simplicial version of the Deligne axioms}
\author{Sebastian Cea}
\date{December 2022}

\begin{document}
\pagenumbering{Roman}
\maketitle
\begin{abstract}
Intersection (co)homology is a way to enhance classical (co)homology, allowing us to use a famous result called Poincaré duality on a large class of spaces known as stratiﬁed pseudomanifolds. There is a theoretically powerful way to arrive at intersection (co)homology by a classifying sheaves that satisfy what are called the Deligne axioms.

We establish an abstract manifestation of the Deligne axioms, to then apply it on a simplicial complex environment, for a category of simplicial sheaves inspired on the works of D. Chataur, D. Tanré and M. Saralegi-Araguren.
For a stablished topology on a triangulation of a stratified pseudomanifold, we find a family of sheaves satisfying the simplicial Deligne axioms, giving us a way to construct intersection cohomology from simplicial sheaves.

\end{abstract}

\tableofcontents

\newpage
\pagenumbering{arabic}
\setcounter{page}{1}
\maketitle

\section{Introduction}

Intersection (co)homology is a powerful enhancement of classical (co)homology, which extends Poincar\'e duality to an important class of spaces known as stratified pseudomanifolds. Its birthplace actually comes from the classification of manifolds, since it was noticed that to study manifolds it would be necessary to extend invariants to a category of \textit{manifolds with singularities}.

One particularly important invariant on this matter was the signature, which can be derived from Poincar\'e duality. So, in \cite{GoreskyMacPhersonIH1} M. Goresky and R. MacPherson solved the ambitious question of finding a generalization of (co)homology on a category of manifolds with singularities. These were called stratified pseudomanifolds, and are good enough as to contain algebraic varieties. We refer to \cite{Friedman} for a complete survey on intersection (co)homology.

Basically, we take on a stratified space $X_0 \subseteq X_1 \subseteq \dots \subseteq X_n=X$ which is almost manifolds (in the sense that each $X_k-X_{k-1}$ is a $k$-manifold and $X-X_{n-2}$ is dense in $X$), and locally conelike, a subcomplex $\IpC(X)$ of the chain complex by discriminating chains with a perversity function $\bar{p}$. The (co)homology of this complex is called intersection (co)homology.

It was a huge success. Nowadays, it has an important role in geometry and applications even in representation theory (see for instance \cite{KazhdanLusztigHistory}).

Now, just as it happens with homology, \textit{there are multiple ways to construct intersection cohomology}, for example:
\begin{itemize}
    \item The classical way, presented by M. Goresky and R. MacPherson in \cite{GoreskyMacPhersonIH1}.
    
    \item In \cite{King}, King gives a \textit{singular} (hence non-simplicial) construction of intersection homology.
    
    \item In \cite{Pollini}, G. Pollini gives a construction of intersection homology using \textit{differential forms}. 
    
    \item In \cite{BlownupDel}, D. Chataur, M. Saralegui-Araguren and D. Tanr\'e construct intersection cohomology with a chain inspired in the \textit{blown up} of (filtered) simplices.
\end{itemize}
There are in fact plenty of models for intersection (co)homology. Given this apparent jungle, one would wish to call upon a characterization up to quasi isomorphism of these functors. 

Thankfully, there are four axioms which characterize the sheaves whose hyper cohomology is intersection cohomology (so we work in the derived category of sheaves complexes). These axioms are called \textit{Deligne's axioms}, in honor to P. Deligne, who had the main idea that would be followed and proved by M. Goresky and R. MacPherson in their second paper on the subject \cite{GoreskyMacPhersonIH2}. The axioms read literally in the original paper as follows

\begin{center}
    Let $S^\cdot$ be a complex of sheaves on $X$, which is constructible with respect to the stratification $\{X_k \}$ and let $S^\cdot_k=S^\cdot|_{X_k-X_{n-k}}$. We shall say $S^\cdot$ satisfies the axioms [AX1] (with perversity $\bar{p}$ with respect to the stratification) provided
    \begin{enumerate}
        \item[(a)] Normalization: $S^\cdot|_{X-\Sigma} \simeq F[n]$, where $F$ is a local system on the regular strata $X-\Sigma$.
        \item[(b)] Lower bound: $H^i(S^\cdot)=0$ for all $i<-n$.
        \item[(c)] Vanishing condition: $H^m(S^\cdot_{k+1})=0$ for all $m>p(k)-n$.
        \item[(d)] Attaching: $S^\cdot$ is $p(k)-n$ attached across each stratum of codimension $k$, i.e., the attaching maps
        $$ H^m(j_k^\star S^\cdot_{k+1}) \to H^m(j_k^\star R(\iota_k)_\star \iota_k^\star S_{k+1}^\cdot) $$
        are isomorphisms for all $k \geq 2$ and all $m \leq \bar{p}(k)-n$.
    \end{enumerate}
\end{center}
They prove that $I^pC^\cdot$, the chain complex whose homology is intersection homology, naturally defines a complex of sheaves that satisfies these axioms for the constant sheaf $F=\mathbb{R}_{X-X_{n-2}}$. Furthermore, they prove that
\begin{center}
    $S^\cdot$ satisfies [AX 1] $\iff$ $S^\cdot \simeq \tau_{\leq {\bar{p}}(n)-n}R (\iota_n)_\star \dots \tau_{\leq {\bar{p}}(2)-n}R(\iota_2)_\star F[n]$
\end{center}
Where the previous isomorphism considered at the level of the derived category. This result is very important as \textbf{it characterizes up to quasi isomorphism all the complexes of sheaves whose hyperhomology is intersection homology}, furthermore giving a quasi isomorphism to the so called \textit{Deligne sheaf}, which is given by the formula $ \tau_{\leq {\bar{p}}(n)-n}R (\iota_n)_\star \dots \tau_{\leq {\bar{p}}(2)-n}R(\iota_2)_\star \mathbb{R}_{X-X_{n-2}}[n]$.

The Deligne axioms are an important cornerstone for developments that enrich the intersection homology theory. Some examples of why we want them are the following.

\begin{itemize}
    \item It provides an easy to check list to see if one's complex of sheaves computes intersection homology. Furthermore, it gives a \textit{map} between different models (sheaves) for intersection homology and the intersection sheaf.
    
    \item Using sheaf theory provides an alternative way to prove many properties. For example it was first proven in \cite{GoreskyMacPhersonIH2} using sheaves that intersection homology is a topological invariant, before it was proven geometrically much later by H. King in \cite{King}.
    
    \item The use of sheaf theory also provides a way to study properties of intersection homology from a local point of view.
    
    \item The shape of the intersection sheaf $\tau_{\leq {\bar{p}}(n)-n}R (\iota_n)_\star \dots \tau_{\leq {\bar{p}}(2)-n}R(\iota_2)_\star \mathbb{R}_{X-\Sigma}[n]$ makes one think that we can extend information from the non-singular strata. This intuition is in fact correct.
\end{itemize}

In the later years there has been signs that point out to the necessity of abstracting the Deligne axioms, and state them in different contexts, in particular for simplicial structures. As the reader probably knows, one can use simplicial structures such as simplicial complexes, simplicial sets, delta sets and cellular complexes, to greatly simplify the computations of (co)homology and other invariants. One can even state that homotopy theory can be worked simplicially. The trick is to actually define the invariants on the corresponding simplicial category, and then proving that it is equivalent to the topological one via a nice functor, so we have for example
$$ H^\D_n(X) \simeq H_n(|X|) $$
The same trick can be done for intersection (co)homology, and this is exposed and expanded in great detail in the works of D. Chataur, M. Saralegui-Aranguren and D. Tanré (see \cite{IHRationalHomotopy}, \cite{SimplicialIH} and \cite{BlownupDel}). 

Now, just as it happens with homology, \textit{there are multiple ways to construct \textbf{simplicial} intersection cohomology}.

\begin{itemize}
    \item Intersection cochain complex $IC^p_\D$, constructed by discriminating simplices $\Dm \too{\s} X \to \Dn$ by the dimension of their filtration.
    \item The \textit{simplicial} blown up cochains, which are constructed in \cite{BlownupIH}.
 
    \item A generalization of the previous functor, extending into considering Sullivan polinomials, was done on \cite{IHRationalHomotopy}.
\end{itemize}

Opening the question for \textit{simplicial} Deligne axioms. Furthermore, the works done on \cite{EinfAlg} point to the goodness of having Deligne axioms in a $E_\infty$-algebra environment.

The Deligne axioms, as we show, can be stated in very general contexts without much problems, allowing us to move between different environments, giving us the first important statement of the paper.

\vspace{0.2 cm}
\textbf{Abstract Deligne Axioms:}

Given a chain of morphisms of categories
\begin{equation} \label{equation chain of functors}
    \begin{tikzcd}
    \C_0 \arrow[shift left=1]{r}[above]{\iota_0} & \C_1 \arrow[shift left=1]{l}[below]{j_0} \arrow[shift left=1]{r}[above]{\iota_1} & \dots \arrow[shift left=1]{l}[below]{j_1} \arrow[shift left=1]{r}[above]{\iota_{n-1}} & \C_n \arrow[shift left=1]{l}[below]{j_{n-1}} \arrow[shift left=1]{r}[above]{\iota_n} & \C_{n+1}=\C \arrow[shift left=1]{l}[below]{j_{n}}
    \end{tikzcd}
\end{equation}
And for a function $p: \mathbb{N} \to \mathbb{Z}$, endofunctors $\tau_{p(k)}^i: \C_i \to \C_i \text{ } k \in \{0, \dots , n \},i \in \{ k,k+1 \}$, satisfying
\begin{enumerate}\renewcommand{\theenumi}{\roman{enumi}}
    \item $j_k \circ \iota_k \simeq 1_{\C_k}$ for all $k \in \{0, \dots, n \}$
    
    \item $\tau_{p(k+1)}^{k} \circ \tau_{p(k)}^{k} \simeq \tau_{p(k)}^{k}$ for all $k$
    
    \item $\tau_{p(k)}^k \circ \tau_{p(k)}^k \simeq \tau_{p(k)}^k$ for all $k$
    
    \item $j_k \circ \tau_{p(k)}^k \simeq \tau_{p(k)}^{k-1} \circ j_k$
\end{enumerate}
We name $j^k=j_k \circ \dots \circ j_n : \C \to \C_k$ for $k \in \{0, \dots , n \}$ and $j^{n+1}=1_\C : \C \to \C$. 

\begin{definition}
In the setting just described, given $F_0 \in \C_0$ such that $\tau_{p(0)}F_0 \simeq F_0$ and $p: \mathbb{N} \to \mathbb{Z}$, we say that $A \in \C$ is of class $Del_{F_0}$ if it satisfies
\begin{enumerate}
    \item[(AX 1)] $j^0A \simeq F_0$ 
    \item[(AX 2)] $j^{k+1}A \simeq \tau_{p(k)} \iota_k j^k A$ $\forall k \in \{0, \dots, n \}$
\end{enumerate}
\end{definition}

And in this context we prove the following proposition.
\begin{proposition} \label{prop 111}
$A \in \C$ is of class $Del_{F_0} \Leftrightarrow A \simeq \tau_{p(n)} \iota_n \dots \tau_{p(0)} \iota_0 F_0$
\end{proposition}

The tricky part is to ground the Deligne axioms by actually finding chain of sheaves that satisfy the axioms. There is one fruitful environment in which we can stablish a family of chains of sheaves that satisfy the axioms, namely, we take a simplicial complex that is a triangulation of a PL space and give it a particular topology. This is fruitful as it actually gives us novel ways to compute intersection cohomology, as the main results of the paper show. The environment goes as follows. 

\begin{enumerate}
    \item We take $X$ a simplicial complex such that $|X|$ is a PL stratified pseudomanifold with stratification $|X_0|\leq \dots \leq |X_n|=|X|$, with $X_k$ is a subcomplex of $X$.
    \item We suppose that $X$ is subdivided two times, $X=Sd^2(X')$.
    \item We take $\T \subset \Sub(X)$ the topology generated by $\{ \underset{\s < \tau}{\bigcup} \st(b_\tau)| \s \in X' \}$, where $b_\tau$ corresponds to the barycenter of $\tau$ and $\st(b_\tau)$ to its star.
    \item We take $U_k^\D = X -_\D X_{n-k}=\{ \s \in X | |\s| \cap X_{n-k}= \emptyset \}$.
\end{enumerate}
For $p: \mathbb{N} \to \mathbb{Z}$ a GM perversity, we take the usual truncation functors as $\tau_{p(k)}$ and as for the chain of functors in (\ref{equation chain of functors}) we take the derived functors of $(\iota_k^\star,(\iota_k)_\star)$, with $U_k \hookar{\iota_k} U_{k+1}$. The simplicial Deligne axioms then read as follows

\begin{definition}
$F \in D_\D(\mathcal{T})$ is said to satisfy the $\D$-Deligne axioms if
\begin{enumerate}
    \item[(AX 1)] $ F|_{U_{0}^\D} \simeq \mathbb{R}_{U_{0}^\D}$ 
    \item[(AX 2)] $F|_{U_{k+1}^\D} \simeq \tau_{p(k)} (\iota_k)_\star F|_{U_{k}^\D}$ $\forall k \in \{0, \dots, n \}$
\end{enumerate}
\end{definition}
And we have as a consequence of proposition \ref{prop 111}, this characterized by the Deligne sheaf.
\begin{proposition} \label{prop 112}
The sheaf $F \in D_\D(\mathcal{T})$ satisfies the $\D$-Deligne axioms if and only if $F \simeq \tau_{p(n)} \iota_n \dots \tau_{p(0)} \iota_0 \mathbb{R}_{U_{0}^\D}$.
\end{proposition}

\vspace{0.2 cm}
\textbf{The main results:} 

We find sheaves that satisfy the simplicial Deligne axioms. Thanks to our functor $\Phi:\Sh(|X|) \to \Sh_\D^\T(X)$, which takes soft sheaves to flasque ones (proposition \ref{proposition Phi soft to flasque}), we have plenty of those.

We state our main results, which both come as corollaries of proposition \ref{prop 112} and the result that is proven in the last section.
\begin{center}
$\Phi(I^pC_\cdot)$ satisfies the $\D$-Deligne axioms.
\end{center} 

We have two consequences of this, the first one is that $\Phi$ takes Deligne sheaves to $\D$-Deligne ones.

\begin{theorem} \label{teo 12}
Given $X$ a simplicial complex such that $|X|$ is a PL stratified pseudomanifold with stratification $|X_0|\leq \dots \leq |X_n|=|X|$, with $X_k$ is a subcomplex of $X$. We suppose that $X$ is subdivided two times, $X=Sd^2(X')$, and take the topology generated by $\{ \underset{\s < \tau}{\bigcup} \st(b_\tau)| \s \in X' \}$. In this setting we have
\begin{center}
    If $F \in D(|X|)$ satisfies the Deligne axioms, then $\Phi(F)$ satisfies $\D$-Deligne axioms.
\end{center}
\end{theorem}

Now, since $\IpC$ is soft we have that $\Phi(\IpC)$ is flasque, and this implies that
$$\mathbb{H}^\star(X,\Phi(\IpC)) \simeq IH^{\bar{p}}(|X|) $$
That is, the hyperhomology of this sheaf corresponds to the intersection homology of the realization. This is a fact shared by all the sheaves that satisfy the $\D$-Deligne axioms.

\begin{theorem} \label{corollary 121}
Consider a PL stratified pseudomanifold $X$ with a compatible triangulation $|K| \too{\simeq} X$. Take $K'=Sd^2(K)$ and the topology $\T$ generated by $\{ \underset{\s < \tau}{\bigcup} \st(b_\tau)| \s \in K \}$ on $K'$. If $F \in D_\D(\T)$ satisfies the $\D$-Deligne axioms, then
$$ \mathbb{H}^\star(K',F) \simeq IH_p^\star(X) $$
\end{theorem}
Giving us ways to compute intersection cohomology.

\section{Background knowledge on Simplicial Complexes} \label{section simplicial complexes}

Given that the main object of study for us are PL stratified pseudomanifolds, which are spaces endowed with a family of triangulations, it will be pleasant to have a good presentation of simplicial complexes. We work with two classical presentations (abstract and topological). We will see a slim view of the abstract presentation, together with some operations of interest, and leave the topological as reference, since all the results we need there can be found in \cite{RourkeSanderson}. 





\begin{definition}
A\textbf{ (abstract) simplicial complex} $K$ is the information of a set (of vertices) $V(K)$ and a set (of simplices) $S(K) \subset \mathcal{P}(V(K))-\emptyset$ such that
\begin{enumerate}
    \item $\{ v \} \in S(K) $ for all $v \in V(K)$
    \item If $\s \in S(K) $ and $\tau \subset \s$ then $ \tau \in S(K)$
\end{enumerate}
\end{definition}

Basically we encapsulate topological information in a combinatorial way. The idea is that each set $\s \in S(K)$ corresponds to a simplex of $K$ and the elements of $\s \subset V(K)$ are its vertices. We say that $\tau$ is a \textbf{face} of $\s$ (written as $\tau < \s$) if $\tau \subset \s$. We abuse notation by calling $K$ to $S(K)$.



Observe that the face relation is a partial order. We call the category associated with it the \textbf{category of simplices} of $K$, and denote it $\D(K)$. We also call $K_n=\{ \s \in K | \text{ } \sharp \s=n+1 \}$ the simplices of dimension $n$ of $K$.

\begin{definition}
    Given $K$ and $L$ simplicial complexes.
    \begin{itemize}
        
        \item A \textbf{morphism of simplicial complexes} $K \too{f} L$ is a function $f:V(K) \to V(L)$ such that $f(\s) \in L$ for all $\s\in K$.
        \item A subset $Y \subset K$ is a \textbf{subcomplex} if for all $\s \in Y$ $\tau \subset \s \Rightarrow \tau \in Y$.
    \end{itemize}
    
\end{definition}

Observe that given a morphism $K \too{f} L$, we can define a function $f:K \to L$ by taking the set image $\s \mapsto f(\s)$, which will satisfy $\tau < \s \Rightarrow f(\tau)<f(\s)$. We will usually work with this function. 

Observe also that if $Y \subset K$ is a subcomplex, we are making $Y$ into a complex by taking $V(Y)= \bigcup_{\s\in Y} \s$. We name $\Sub(K)$ the set of subcomplexes of $K$, they form a category under inclusion. The following proposition shows that some usual constructions in $\Set$ form subcomplexes.

\begin{proposition} \label{Set operations on Sset}
Let $K$ be a simplicial complex, we have that

\begin{enumerate}
    \item For $\{ Y^i \}_{i \in I}$ a family of subcomplexes of $K$, 
    \begin{itemize}
        \item $\bigcup_{i \in I} Y^i \leq K$ with $(\bigcup_{i \in I} Y^i)_n=\bigcup_{i \in I} Y^i_n$
        \item $\bigcap_{i \in I} Y^i \leq K$ with $(\bigcap_{i \in I} Y^i)_n=\bigcap_{i \in I} Y^i_n$
    \end{itemize}
      
    \item For a morphism $f:K \to L$ of simplicial complexes, $\Ima(f) \leq L$
    
    \item For a morphism $f:K \to K' $ of simplicial complexes, and subcomplexes $Z \leq K$ and $Z' \leq K'$ we have that $f^{-1}(Z') \leq K$ and $f(Z)=\Ima(f \circ \iota_Z) \leq K'$ (where $\iota_Z: Z \hookrightarrow K$ is the inclusion)
\end{enumerate}
\end{proposition}
From $1.$ and $2.$ of this proposition we have an interesting corollary.

\begin{corollary} \label{Sub(X) topological space}
$(K,\Sub(K))$ is a topological space.
\end{corollary}

And we also have of course many topological spaces whose open sets are not all subcomplexes. Let us see a proof of the proposition now.

\begin{proof}

\begin{enumerate}
    \item Given $\s \in \bigcup_{i \in I} Y^i$, say $\s \in Y^{i_0}$. Then for any $\tau < \s$, since $Y^{i_0} \leq K$, we have $\tau \in Y^{i_0} \subset \bigcup_{i \in I} Y^i$, and $(\bigcup_{i \in I} Y^i)_n=K_n \cap \bigcup_{i \in I} Y^i = \bigcup_{i \in I} K_n \cap Y^i= \bigcup_{i \in I} Y^i_n$. The proof for the intersection is analogous.
    \item Consider $f(\s)=\{f(\s_0), \dots , f(\s_m) \}\in \Ima(f)$ (here $\s_i\in \s$ $\forall i$), then if $\tau \subset f(\s)$, we have that $\tau= \{ f(\s_{i_0}), \dots , f(\s_{i_k}) \}=f(\{\s_{i_0}, \dots , \s_{i_k} \})$, and $\{\s_{i_0}, \dots , \s_{i_k} \} \in K$ since $K$ is a simplicial complex.
    \item If $f(\s) \in Z'$ and $\tau<\s$ then $f(\tau)<f(\s)$ by the previous observation and since $Z' \leq X$ then $f(\tau) \in Z'$. The second part of this statement is a consequence of $2.$
\end{enumerate}
\end{proof}

Even when a subset $S \subset K$ is in general not a subcomplex, we can consider the subcomplex generated by it
$$ <S> = \underset{S \subset Y \leq K}{\bigcap} Y $$
Which is simply $S$ plus the subsets of elements of $S$. We name $ \ims = <\{\s\}>  $. Observe that this is just $\mathcal{P}(\s)$. Abstractly we name $ \Dn = \mathcal{P}(\{0, \dots, n \}) $ (so $\ims \simeq \D^{\texttt{dim}(\s)}$)

We make some very easy to prove observations
\begin{proposition}
    Given $L\leq K$ complexes, $f:K \to K'$ a morphism and $\s,\tau \in K$
    \begin{enumerate}
        \item $\s \in Y$ if and only if $\ims \leq Y$
        \item $\s< \tau \iff \ims \leq \imt$
        \item $Y=\underset{\s\in Y}{\bigcup}\ims$
        \item $f(\ims)=\Ima(f(\s))$
    \end{enumerate}
\end{proposition}
\begin{proof}
    The first statement comes from the fact that $Y$ is a simplicial complex. The second statement is because $\s <\tau \Leftrightarrow \s \in \imt \Leftrightarrow \ims \leq \imt$. The third statement follows directly from the first. The last statement comes from elemental set theory.
\end{proof}
We can also define the product of simplicial complexes as the categorical one, giving us $|K \times L|=|K|\times |L|$. This product will have a nice presentation if we order the vertices.

\begin{definition}
Given $K$ and $L$ simplicial complexes with their vertices ordered. We define $L \times K$ as follows
\begin{itemize}
    \item $V(K \times L) = V(K) \times V(L)$
    \item $S(K \times L) = \{ \{(v_0,w_0), \dots ,(v_n,w_n) \} | \{v_0, \dots ,v_n \} \in S(K), \{w_0, \dots ,w_n \} \in S(L) \text{ and } v_0  \leq \dots \leq v_n , w_0 \leq \dots \leq w_n \}  \} $
 \end{itemize}
 \end{definition}

The operation that sadly does not give a subcomplex in a nice way .is the subtraction of sets. This is a very important operation for us since the Deligne axioms work greatly with subtractions (recall that we take $U_k= X-X_{n-k}$). We will replace this operation for the closest posible one.
$$ Y-_\D Z= \{ \s\in Y | Z \cap \ims = \emptyset \} = \{ \s \in Y | V(Z)\cap \s = \emptyset \} $$
Many properties of set subtraction are also true for this simplicial subtraction.

\begin{proposition} \label{prop -D}
Given $Y,Y',Z,Z' \leq X$ simplicial complexes.

\begin{enumerate}
    \item $Y-_\D Z= \bigcup \{ W \subset Y-Z| W \leq X \}  \leq Y$
    \item $Y-_\D \emptyset= Y$ and $Y-_\D Y= \emptyset$ (this is a version of $X^c= \emptyset$ and $\emptyset^c=X$)
    \item If $Z \leq Z'$ then $Y -_\D Z' \leq Y -_\D Z$
    \item $Y -_\D Z= Y \cap (X-_\D Z)$
    \item Morgan Laws: 
    \begin{enumerate}
        \item $X -_\D (Y \cup Z)= (X-_\D Y) \cap (X -_\D Z)$
        \item $(X-_\D Y) \cup (X -_\D Z) \leq X -_\D (Y \cap Z)$
    \end{enumerate}
    \item $Z -_\D Y \leq (X -_\D Y)  -_\D (X -_\D Z) $ (in particular $Z \leq X -_\D (X -_\D Z)$. Equality does not hold as we will see)
    \item If $Y \leq Z$ then $Y -_\D Z= \emptyset$
    \item If $Z' \leq Z$ then $Y -_\D Z= (Y -_\D Z') -_\D Z $
    \item $(X-_\D Y)-_\D Z = X -_\D (Y \cup Z)$
    \item Others
    \begin{enumerate}
        \item $Y \cap (Z -_\D Z')=(Y \cap Z) -_\D (Y \cap Z')$
        \item $(Z \cup Z') -_\D Y=(Z -_\D Y) \cup (Z' -_\D Y)$
        \item $(Z \cap Z') -_\D Y = (Z -_\D Y) \cap (Z' -_\D Y)$
        \item $Y -_\D (Y \cap Z)= Y -_\D Z=(Y \cup Z) -_\D Z$
        \item $(Y -_\D Y') \cap (Z -_\D Z') = (Z -_\D Y') \cap (Y -_\D Z') $
    \end{enumerate}
\end{enumerate}
\end{proposition}

\begin{proof}
\begin{enumerate}
    \item It is obvious that $Y -_\D Z \subset Y$. To see that it is a simplicial set notice that if $\s \in Y -_\D Z$ and $\tau < \s$ then $\Ima(\tau) \leq \ims$ and then $\ims \cap Z \subset \Ima(\tau) \cap Z = \emptyset$.

We see now that $Y -_\D Z$ is optimal. It is obvious that $Y -_\D Z \subset Y-Z$. Furthermore if we have a simplicial set $W \subset Y-Z$ and $\s \in W$, we will have that $\ims \leq W \subset Y-Z$ and therefore $\ims \cap Z= \emptyset$ and then $\s \in Y -_\D Z$.

\item $Y-_\D \emptyset = \{ \s \in Y | \ims \cap \emptyset= \emptyset \}=Y$ and $Y-_\D Y \subset Y- Y=\emptyset$

\item This comes from $\ims \cap Z \subset \ims \cap Z' = \emptyset$

\item This comes directly from the definitions.

\item The first Morgan law comes from the fact that $\ims \cap (Z \cup Y)=(\ims \cap Z) \cup (\ims \cap Y)$ is empty if only if $\ims \cap Y$ and $\ims \cap Z$ are empty. As for the second, if let us say $\ims \cap Y$ is empty, then $\ims \cap Y \cap Z$ is empty.

\item Given $\s \in Z$ with $\ims \cap Y=\emptyset$, we need to prove that $\ims \cap (X-_\D Z)=\emptyset$. Say that $\tau \in \ims \cap (X-_\D Z)$, since $\s \in Z$ then $\ims \leq Z$ and $\tau \in \ims \Rightarrow \imt \leq \ims \leq Z$, which contradicts that $\tau \in X-_\D Z$ (notice that $\imt$ is not empty since $\tau \in \imt$)

\item This comes from the fact that $\s \in Z$ implies that $\ims \cap Z \neq \emptyset$ (notice that $\ims$ is not empty since $\s \in \ims$)

\item Since $Y -_\D Z' \leq Y$ we have that $ (Y -_\D Z') -_\D Z \leq Y -_\D Z $. On the other direction, given $\s \in Y -_\D Z$, we have that $\s \in Y$ and $ \ims \cap Z = \emptyset$, which implies that $\ims \cap Z'= \emptyset$ since $Z' \leq Z$, and these three last properties mean that $\s \in (Y -_\D Z') -_\D Z$.

\item This comes from the fact that $\ims \cap (Z \cup Y)= (\ims \cap Z) \cup (\ims \cap Y)= \emptyset$ if and only if $\ims \cap Z = \emptyset$ and $\ims \cap Y= \emptyset$.

\item For statement (a), one inclusion comes from $\ims \cap (Y \cap Z') \subset \ims \cap Z'$ and the other inclusion follows from the fact that if $\s \in Y \cap Z \subset Y$ then $\ims \leq Y$ and therefore $\ims \cap Y \cap Z'= \emptyset \Rightarrow \ims \cap Z'= \emptyset$.

Statemens (b), (c) and (e) are straightforward. For the first equality of (e), one inclusion is a consequence of 3. using $Y \cap Z \leq Z$ and the second inclusion comes from the fact that if $\s \in Y$ then $\ims \leq Y$ and then $\ims \cap Y \cap Z= \emptyset$ implies $\ims \cap Z=\emptyset$. For the second equality $Y -_\D Z \leq (Y \cup Z) -_\D Z$ is obvious and for the other inclusion notice that if $\ims \cap Z=\emptyset$ then $\s$ is not in $Z$ since $\s \in \ims$.

\end{enumerate}

\end{proof}

In other words, $ -_\D $ functions as $-$ except for three important details, which is that in general we are missing the following
\begin{enumerate}
    \item $X -_\D (Y \cap Z) \leq (X-_\D Y) \cup (X -_\D Z) $
    \item $X -_\D (X -_\D Z) \leq Z$
    \item $Y -_\D Z= \emptyset \Longrightarrow Y \leq Z$
\end{enumerate}

A counterexample for the Morgan law is to take $X= \D^1$ and Y and Z as the two points of its boundary $\partial \D^1$. As for the other two properties $\partial \D^2 \leq \D^2$ provide counterexamples.

\begin{center}
\includegraphics[scale=0.6]{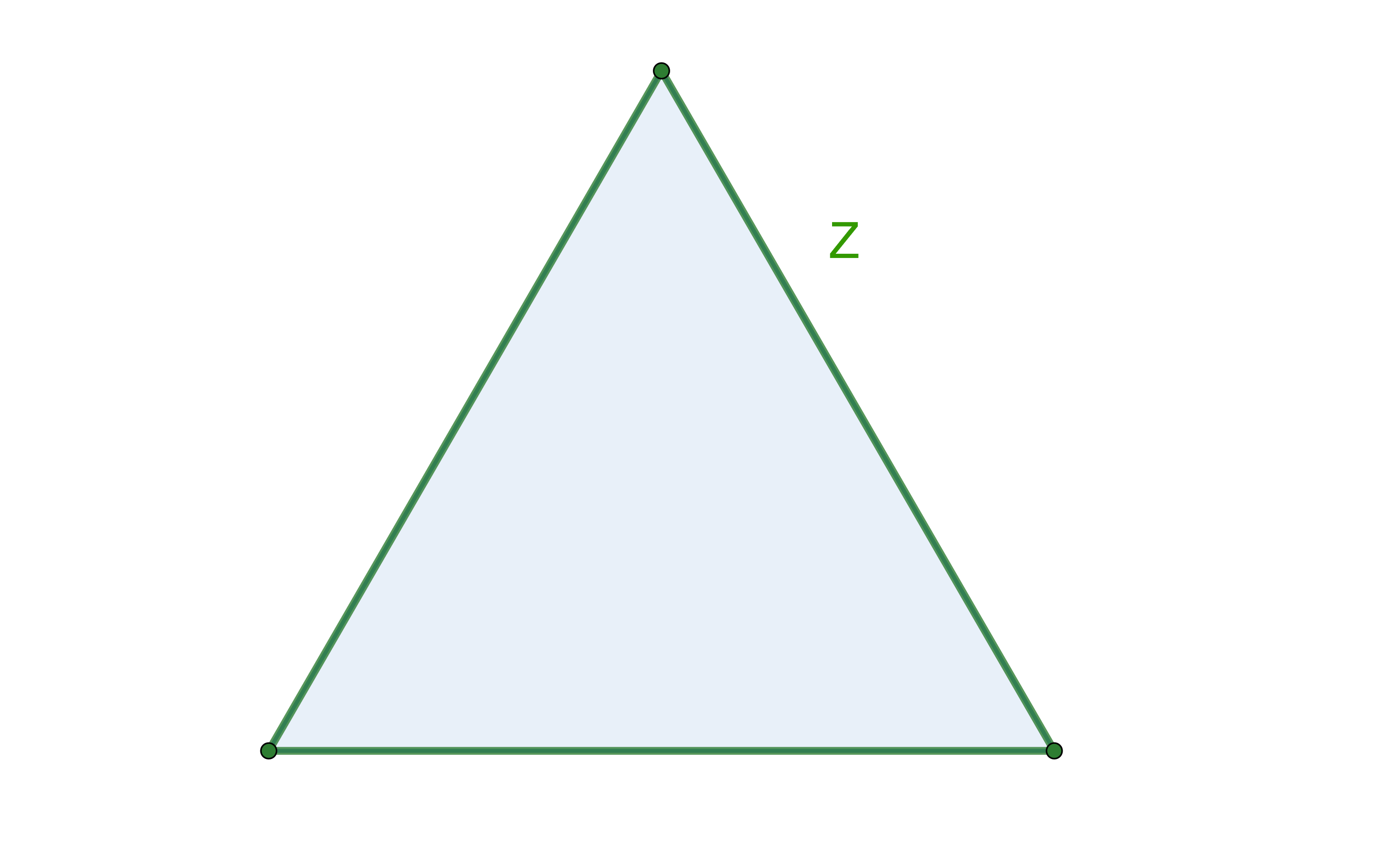}
\end{center}

The problem is that when we perform $X -_\D Z$, together with $Z$ we are also taking off an aura around it (technically a star), which in the example $\partial \D^2 \leq \D^2$ will take off all $\D^2$. This problem can be solved by \textit{subdividing}.

\begin{center}
\includegraphics[scale=0.6]{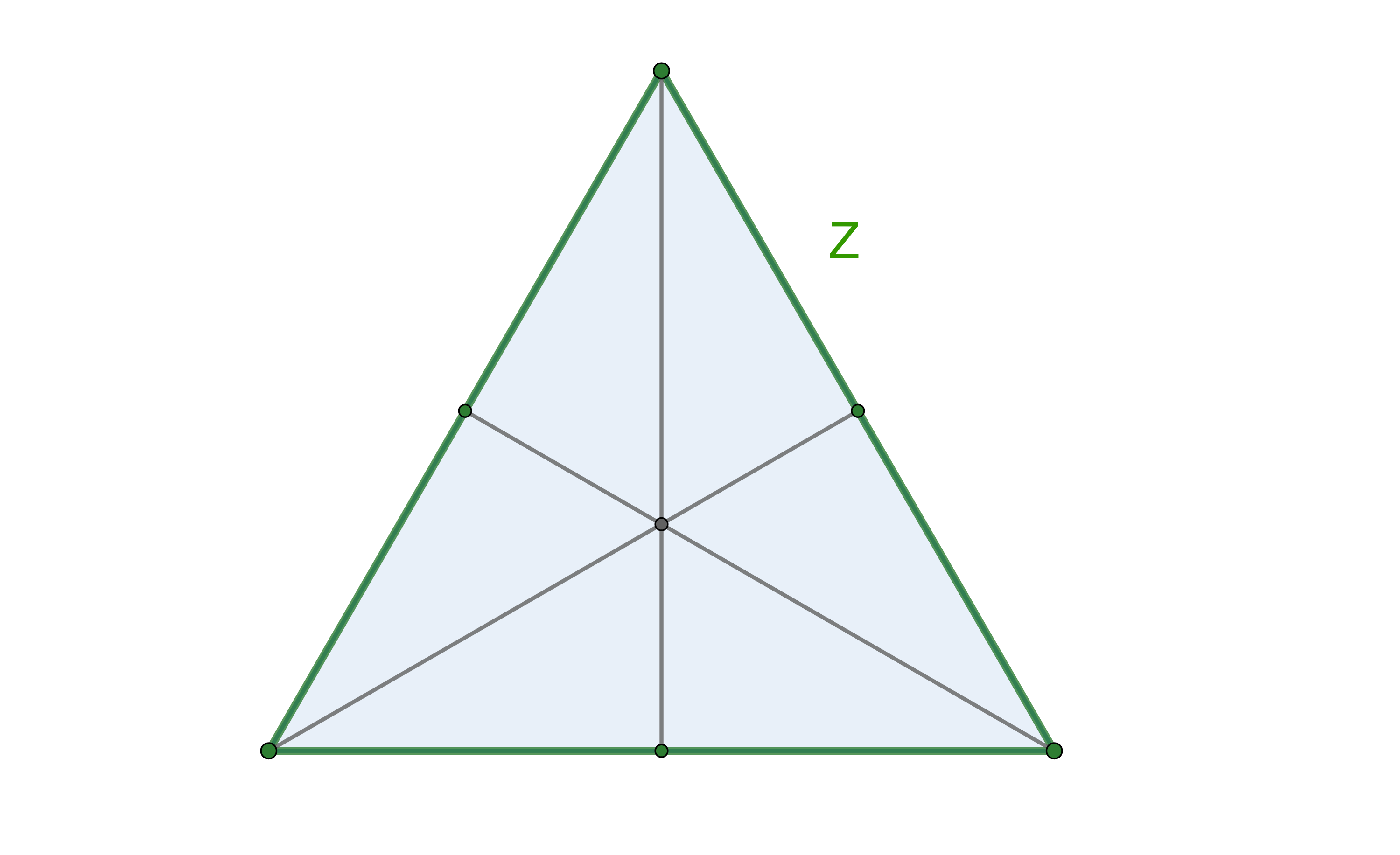}
\end{center}

Before showing how this is done, we define a special construction on simplicial complexes, which will come handy on the proof of our main theorem.

\begin{definition}\label{Def G}
Given $K$ a simplicial complex and a set of vertices $A \subset V(K)$, we define $\mathcal{G}(A) \leq K$ as
\begin{itemize}
    \item $V(\G(A))=A $
    \item $S(\G(A))= \{ \s \in S(K) | \s \subset A \}$
\end{itemize}
\end{definition}
Since $\subset$ is transitive we will have that $\G(A)$ is a simplicial complex. We present some properties of this construction

\begin{proposition}\label{prop G}
Given $K$ a simplicial complex, the construction above defines a function $\G: \mathcal{P}(V(K)) \to \Sub(K)$ that satisfies
\begin{enumerate}
    \item $V(\G(A))=A$
    \item $A =B \Leftrightarrow \G(A) = \G(B)$ (that is, $\G$ is injective)
    \item $A \subset B \Rightarrow \G(A) \subset \G(B)$
    \item $Z \leq \G(V(Z))$ for all $Z \leq K$
    \item $ \underset{i \in I}{\bigcap} \G(A_i) = \G(\underset{i \in I}{\bigcap} A_i)$
    \item $ \underset{i \in I}{\bigcup} \G(A_i) \leq \G(\underset{i \in I}{\bigcup} A_i)$
    \item For $f: K \to L$ a morphism, $\G(f(A)) \leq f(\G(A))$
    \item For $f: K \to L$ a morphism, $\G(f^{-1}(A)) = f^{-1}(\G(A))$
    \item $\G(A \times B)= \G(A) \times \G(B)$
    \item $\G(A-B)=\G(A) -_\D \G(B)$
\end{enumerate}
\end{proposition}
\begin{proof}
The first statement is definition, and the second statement comes from this since $\G(A)=\G(B) \Rightarrow A= V(\G(A))=V(\G(B))=B$. The third statement is true because if $\s \in \G(A)$ then $\s \subset A \subset B$ so $\s \in \G(B)$. The forth statement is obvious as all simplices of $Z$ have vertices in $Z$.

The fifth statement comes from the set theoric property that $\s \in \bigcap A_i \Leftrightarrow \s \in A_i$ $\forall i$, and similarly the sixth statement comes from $\s \in A_i$ for some $i$ then $\s \in \bigcup A_i$. For the seventh statement, if we have $f(\s) \in f(\G(A))$, $\s \in \G(A)$ means that $\s \subset A$ and then $f(\s) \subset f(A)$ so $f(\s) \in \G(f(A))$. As for the eighth statement, $\s \in \G(f^{-1}(A)) \Leftrightarrow \s \subset f^{-1}(A) \Leftrightarrow f(\s) \subset A \Leftrightarrow f(\s) \in \G(A) \Leftrightarrow \s \in f^{-1}(\G(A)) $.

For the ninth statement, consider $\s= \{ (v_0,w_0), \dots , (v_n,w_n) \} \in K \times K'$ (that is, $\fset{v} \in K$, $\fset{w} \in K'$ and $v_i \leq v_{i+1}$, $w_i \leq w_{i+1}$ for all $i$). If $\s \in \G(A) \times \G(B)$ then $\fset{v} \in \G(A)$ and $\fset{w} \in \G(B)$, so $(v_i,w_i) \in A \times B$ for all $i$, so $\s \in \G(A \times B)$. On the other direction, if $\s \in \G(A \times B)$ then $v_i \in A$ and $w_i \in B$ for all $i$ which means that $\fset{v} \in \G(A)$ and $\fset{w} \in \G(B)$. We conclude that $\s \in \G(A) \times \G(B)$.

For the last statement, given $\s \in \G(A) -_\D \G(B)$ we have that $\s \in \G(A)$ so $\s \subset A$, and we also have that $\s \cap V(\G(B))=\s \cap B=\emptyset$. This means that $\s \subset A-B$.

For the other direction, consider $\s \in \G(A-B)$, this is $\s \subset A-B$, which implies that $\s \subset A$ (so $\s \in \G(A)$ and also it means that $\s \cap B= \s \cap V(\G(B))= \emptyset$.
\end{proof}

Of course we will not always have that $Z= \G(Z)$, as for example $\partial \Dn \leq \Dn$. We give a special name for complexes in the image of $\G$
\begin{definition}
A simplicial complex $Z \leq K$ is called \textbf{fat} if $Z=\G(V(Z))$
\end{definition}
A good example of a fat subsimplicial complex is $\ims$, since
$$\ims = \G(\s) $$
The proof of this is simply that $\tau \in \G(\s) \Rightarrow \tau \subset \s \Rightarrow \tau < \s$. Observe that this means that $\mathcal{G}(V(Z)) \cap \ims= \mathcal{G}(\s \cap V(Z)) = \mathcal{G}(\tau)= \Ima(\tau)$ for some $\tau < \s$. We will state this as a corollary.

\begin{corollary} \label{corollary 312}
For any $\s \in K$ and fat subcomplex $Z=\mathcal{G}(V(Z)) \leq K$ we have that
$$Z \cap \ims = \Ima(\tau) $$
For some $\tau< \s$.
\end{corollary}

We also present a lemma that will use in the proof of the main result. If one is not careful at all, one might be tempted to think that a bijection $f: A \to B$ defines an isomorphism $f: \G(A) \to \G(B)$. This is of course false as $\G(A)$ and $\G(B)$ depend on the complexes in which they are defined. However, we do have the following lemma.

\begin{lemma} \label{prop G2}
Given $K$ and $L$ simplicial complexes and consider $A \subset V(K)$ and $B \subset V(L)$. Consider a bijection $f: A \to B$ with inverse $g: B \to A$

If $f: \G(A) \to \G(B)$ and $g: \G(B) \to \G(A)$ define morphisms, then they are inverse of each other. In particular
$$f: \G(A) \too{\simeq} \G(B) $$
is an isomorphism.
\end{lemma}

\begin{proof}
With this number of hypothesis one would expect that the proof is very easy, and it is. It is just the fact that $f(g(\s))=f(f^{-1}(\s))=\s$ and $g(f(\s))=g(g^{-1}(\tau))=\tau$ (here $f^{-1}$ and $g^{-1}$ refer to the preimage).
\end{proof}

We will now define the barycentric subdivision.

\begin{definition}
Given a simplicial complex $K$ we define its \textbf{subdivision} $Sd(K)$ as the following (abstract) simplicial complex
\begin{itemize}
    \item $V(Sd(K))= S(K)$
    \item $S(Sd(K))= \{ \fset{\s} \in \mathcal{P}(S(K)) | \s_0 \subset \dots \subset \s_n \} $
\end{itemize}
\end{definition}

We call the \textbf{barycenter} of $\s \in K$ to the vertex $\{ \s \}$. Now, notice that given $Z \leq K$, the subdivision $Sd(Z)$ defines a subcomplex of $Sd(K)$, that is, we have a functor $Sd: \Sub(K) \to \Sub(Sd(K))$. We will state some properties of this functor.

\begin{proposition} \label{prop Sd}
Given $K \in \textbf{Scx}$, the functor $Sd: \Sub(K) \to \Sub(Sd(K))$ satisfies
\begin{enumerate}
    \item $Z=Y \Leftrightarrow Sd(Z)=Sd(Y)$ (that is, $Sd$ is injective)
    \item $Z \leq Y \Rightarrow Sd(Z) \leq Sd(Y)$
    \item $Sd(\bigcup Y_i)= \bigcup(Sd(Y_i))$
    \item $Sd(\bigcap Y_i) = \bigcap Sd(Y_i)$
    \item $\fset{\s} \in Sd(Z) \Leftrightarrow \s_n \in Z$ (where we suppose that $\s_0 \subset \dots \subset \s_n$)
    \item $Sd(Z)= \G(V(Sd(Z))= \G(Z)$ (it is fat)
\end{enumerate}
\end{proposition}
\begin{proof}
The first statement comes from $Z= V(Sd(Z))=V(Sd(Y))=Y$. The fifth statement comes from the fact that for $Z \leq K$ we have $\tau \subset \s \in Z \Rightarrow \tau \in Z$ together with $\s_0 \subset \dots \subset \s_n$. The second statement is clear using the fifth.

The third statement comes as follows $\fset{\s}\in \bigcup Sd(Y_i) \Leftrightarrow \exists i_0$ such that $\fset{\s} \in Sd(Y_{i_0}) \Leftrightarrow \exists i_0$ such that $\s_n \in Y_{i_0} \Leftrightarrow \s_n \in \bigcup Y_i \Leftrightarrow \fset{\s} \in Sd(\bigcup Y_i)$

Similarly $\fset{\s}\in \bigcap Sd(Y_i) \Leftrightarrow \forall i$ $\fset{\s} \in Sd(Y_i) \Leftrightarrow \forall i$ $\s_n \in Y_{i} \Leftrightarrow \s_n \in \bigcap Y_i \Leftrightarrow \fset{\s} \in Sd(\bigcap Y_i)$

For the last statement, consider $\fset{\s} \in \G(V(Sd(Z)))=\G(Z)$, we have that $ \fset{\s} \subset V(\G(Z))=Z $ so $\fset{\s} \in Sd(Z)$
\end{proof}

We will usually abuse notation and call $Z$ to $Sd(Z)$ when there is no risk of confusion. We will also suppose that when writing $\fset{\s} \in Sd(K)$ we have that $\s_0 \subset \dots \subset \s_n$.

The key point of why the subdivision works well with $-_\D$ is the following lemma.




\begin{lemma}
For $X$ simplicial complex and $Y,Z \leq X$, given $\fset{\s} \in Sd(X)$ we have
$$ \fset{\s}\in Sd(Y)-_\D Sd(Z) \Longleftrightarrow \fset{\s} \subset Y-Z $$
\end{lemma}
\begin{proof}
$\fset{\s}\in Sd(Y)-_\D Sd(Z) \Leftrightarrow \fset{\s} \subset Y \text{ and } \fset{\s} \cap Z = \emptyset \Leftrightarrow \s_i \in Y-Z \text{ } \forall i $
\end{proof}
And now the issues we had with $-_\D$ can be solved.

\begin{proposition}
Given $X$ a simplicial complex, we have
\begin{enumerate}
    \item $Sd(X) -_\D (Sd(X) -_\D Sd(Z)) \leq Sd(Z)$
    \item $Sd(Y) -_\D Sd(Z)= \emptyset \Longrightarrow Sd(Y) \leq Sd(Z)$
    \item $Sd(X) -_\D (Sd(Y) \cap Sd(Z)) \leq (Sd(X)-_\D Sd(Y)) \cup (Sd(X) -_\D Sd(Z)) $
\end{enumerate}
\end{proposition}
\begin{proof}
For the first two properties we use that $Sd(Z)=\G(Z)$ for all $Z \leq X$.

So $Sd(X) -_\D (Sd(X) -_\D Sd(Z)) = \G(X) -_\D (\G(X) -_\D \G(Z))= \G(X-(X-Z))=\G(Z)=Sd(Z)$

And $Sd(Y) -_\D Sd(Z)= \G(Y) -_\D \G(Z)= \G(Y-Z) = \emptyset = \G(\emptyset)$, then $Y-Z = \emptyset$ and therefore $Y \subset Z$ which implies $Sd(Y) \leq Sd(Z)$

The Morgan law is slightly more subtle. Consider $\fset{\s} \in Sd(X) -_\D (Sd(Y) \cap Sd(Z)) = Sd(X) -_\D (Sd(Y \cap Z))$, which says that $\forall i$ $\s_i \notin Z \cap Y$.

Now suppose $\fset{\s} \notin Sd(X)-_\D Sd(Y)$, then for some $i_0$ we have that $\s_{i_0}\in Y$, then as $Y \leq X$ $\s_0 , \dots, \s_{i_0} \in Y$ and therefore $\s_0, \dots, \s_{i_0} \notin Z$. This means, since $Z \leq X$, that $s_k \notin Z$ for $k > i_0$, with which we conclude that $\fset{\s}\in Sd(X)-_\D Sd(Z)$
\end{proof}

We finish this section with an abstract definition of the star and the link.

\begin{definition}
Given $X$ a simplicial complex and $\s \in X$
\begin{itemize}
    \item $st(\s)= \underset{\s < \tau}{\bigcup} \Ima(\tau)$
    \item $lk(\s)= st(\s)-_\D \ims$
\end{itemize}
\end{definition}

On the proof of the main result we also make use of the geometrical point of view on simplicial complexes. We use the results and language that can be found in the first two chapters of \cite{RourkeSanderson}.

\section{Simplicial Deligne Axioms} \label{chapter Simplicial Deligne Axioms}

We now move into the main results of this paper, which corresponds to giving a simplicial incarnation of the Deligne axioms, and then providing chains of sheaves that satisfy them. This actually provides new ways of constructing intersection homology, from a combinatorial viewpoint. In our way, we actually abstract the Deligne axioms themselves, giving them a lissom presentation, and conditions to be stated in many contexts.

We work in a context of \textit{simplicial sheaves}, which in our case is actually a category of sheaves over a topology made out of subcomplexes. Working under a suitable topology, we state the simplicial Deligne axioms as a natural mirror of the topological ones, and we use the functor
$$ \Phi: \Sh(|X|) \to \Sh_\D^\mathcal{T}(X) $$
Constructed naturally from the realization functor $|-|$ by taking $\Phi(F)(Y)= \Gamma(|Y|,F)=\texttt{colim}_{|Y| \subseteq U}F(U)$, to find $\D$-Deligne sheaves. This functor is special as it takes soft sheaves into flasque sheaves. In particular $\Phi(\IpC)$ is flasque, and this implies that
$$ \mathbb{H}^\star(X,\Phi(\IpC)) \simeq IH^{\bar{p}}_\star(|X|) $$
We prove that $\Phi(\IpC)$ satisfies the simplicial Deligne axioms, and this brings as a consequence our two main results.
\begin{itemize}
    \item If $F$ is a Deligne sheaf then $\Phi(F)$ satisfies de simplicial Deligne axioms.
    \item If $F$ satisfies the simplicial Deligne axioms, then $\mathbb{H}^\star(X,F) \simeq IH^{\bar{p}}_\star(|X|)$.
\end{itemize}

\subsection{Simplicial Sheaves} \label{section simplicial sheaves}

Although it might not be strictly necessary to say, we want to remark that we are working with \textit{simplicial sheaves}. This means that the domain category of the sheaves is made out of simplices, following a tradition of defining invariants and other mathematical objects on simplicial categories. The most known example of this is the simplification of the calculation of homology by defining for simplices, that is, if $X$ is a triangulation of $|X|$, we take have that
$$ H_\star^\D(X) \simeq H_\star(|X|) $$
Here $H_\star^\D: \textbf{Scx} \to \textbf{Ab}$ is a functor defined on a category of simplicial structures, which is much easier to calculate than $H_\star$. This simplicial homology is the homology of a chain complex made of simplicial chains.
$$ C_n^\D(K)= \mathbb{Z} K_n$$
Observe that for each $n$, the duals of these $C_n^\D$'s define contravariant functors
$$C^n_\D: \textbf{Scx}\op \to \textbf{Ab}$$
Now, if we fix a simplicial complex $X$, we have that the category $\Sub(X)\op$ (corresponding to the poset of subcomplexes of $X$ under inclusion) is a subcategory of $\textbf{Scx}\op$, hence the $C_n^\D$'s can be restricted to it
$$ C^n_\D : \Sub(X)\op \to \textbf{Ab} $$
Corollary \ref{Sub(X) topological space} then means that $C^n_\D$ is actually a presheaf! It is moreover easily seen to be a sheaf. This inspires a theory of \textit{simplicial sheaves}, constructed simply by sheafifying functors $\Sub(X)\op \to \textbf{A}$ (with \textbf{A} being some abelian category). Examples of these sheaves are abundant
\begin{itemize}
    \item The sheafification of the constant functor $\underline{M}:U \mapsto M $ is a simplicial sheaf.
    \item The Sullivan polynomials $A_{PL}$, as constructed in \cite{Sullivan} can be regarded as simplicial sheaves (ordering the vertices).
    \item The Blown up Sheaf constructed in \cite{BlownupIH} can be regarded as a simplicial sheaf (ordering the vertices).
\end{itemize}
And as sheaves, they inherit sheaf properties. For example, we can talk about flasque sheaves, which in the literature are been called \textit{extendable}. 
A reason for which considering simplicial sheaves makes sense is that, just as with topological sheaves, the hypercohomology of the constant sheaf is simplicial cohomology
\begin{itemize}
    \item Consider $\underline{\mathbb{Q}}$ be the constant sheaf as in the example before and $X$ a simplicial complex with its vertices ordered. We want to calculate the hypercohomology $\mathbb{H}^*(X,\underline{\mathbb{Q}})$. We can easily check that $\underline{\mathbb{Q}}$ is not flasque, since for example, as $\partial \Delta^1$ is the discrete category with two points, we have that $$\Gamma(\underline{\mathbb{Q}})(\partial \Delta^1)= \mathbb{Q} \oplus \mathbb{Q}$$
    And there are no surjections $\mathbb{Q} \to \mathbb{Q} \oplus \mathbb{Q}$.
    
    With the result shown in (\cite{FelixHalperinThomasRHT} Lemma 10.7), we directly deduce that $\underline{\mathbb{Q}} \to A_{PL}$ gives a flasque resolution of $\underline{\mathbb{Q}}$. Hence $\mathbb{H}^i(X,\underline{\mathbb{Q}})=H^i(A_{PL}(X))=H^i(X) $. That is to say, the hyperhomology of the constant functor $\underline{\mathbb{Q}}$ is the classical cohomology.
\end{itemize}
In a way, we are extending this from cohomology to intersection cohomology, by the use of Deligne axioms.

We will consider as a reference, the Boorel-Moore intersection (co)homology, as constructed in \cite{Banagl}. The functor $\IpC$ constructed there can be considered as a simplicial sheaf. We will work with this sheaf as our main reference, and we assume all the constructions and results from (\cite{Banagl}, chapters 1-4).

If we fix a triangulation, we can also see $\IpC$ as one of these simplicial sheaves, and the construction made in the last section can be abstracted as follows. For a PL space $X$ with a family of triangulations $\mathbb{T}$, a PL sheaf system is a family of sheaves $\{ F_T \}_{T \in \mathbb{T}}$ with corresponding morphisms of sheaves compatible with subdivisions. We can construct a sheaf on $X$ from there in a way similar to $\IpC$, and work out Deligne axioms with this theory.
Luckily, we do not have to do that. We can obtain our results on one fixed triangulation of $X$. For reasons that will be explained soon, we subdivide twice (spoiler: the subtraction works better like that) and then we take a particular topology made out of subcomplexes. We name the category of sheaves over a topology $\T \subset \Sub(X)$ as $\Sh_\D^\T(X)$ (for topologies smaller than $\Sub(X)$, the examples given before are sheaves by considering the restriction).

\subsection{Abstract presentation for Deligne axioms} \label{abstract deligne axioms}
We give an abstract presentation on Deligne axioms that can serve us into stating this axioms in many contexts, which gives us a particularly lissom presentation.

Consider a chain of categories with morphisms
\begin{equation*}
    \begin{tikzcd}
    \C_0 \arrow[shift left=1]{r}[above]{\iota_0} & \C_1 \arrow[shift left=1]{l}[below]{j_0} \arrow[shift left=1]{r}[above]{\iota_1} & \dots \arrow[shift left=1]{l}[below]{j_1} \arrow[shift left=1]{r}[above]{\iota_{n-1}} & \C_n \arrow[shift left=1]{l}[below]{j_{n-1}} \arrow[shift left=1]{r}[above]{\iota_n} & \C_{n+1}=\C \arrow[shift left=1]{l}[below]{j_{n}}
    \end{tikzcd}
\end{equation*}
Such that $j_k \circ \iota_k \simeq 1_{\C_k}$ for all $k \in \{0, \dots, n \}$

We will name $j^k=j_k \circ \dots \circ j_n : \C \to \C_k$ for $k \in \{0, \dots , n \}$ and $j^{n+1}=1_\C : \C \to \C$. This will act as our "projection to subspaces".

We need two more elements, first a sequence $p: \mathbb{N} \to \mathbb{Z}$ and secondly a "truncation functor", which in this case is a set of endofunctors
$$\tau_{p(k)}^i: \C_i \to \C_i \text{ } k \in \{0, \dots , n \},i \in \{ k,k+1 \}$$
that satisfy the following
\begin{enumerate}\renewcommand{\theenumi}{\roman{enumi}}
    \item $\tau_{p(k+1)}^{k} \circ \tau_{p(k)}^{k} \simeq \tau_{p(k)}^{k}$ for all $k$
    \item $\tau_{p(k)}^k \circ \tau_{p(k)}^k \simeq \tau_{p(k)}^k$ for all $k$
\end{enumerate}
We will omit the superindex whenever the categories involved are clear. We will also ask a compatibility with the $j_k$ functors

\begin{itemize}
    \item $j_k \circ \tau_{p(k)}^k \simeq \tau_{p(k)}^{k-1} \circ j_k$
\end{itemize}

And with these hypothesis we can state the Deligne axioms

Now take $F_0 \in \C_0$ such that $\tau_{p(0)}F_0 \simeq F_0$. Having this setting, we state the Deligne axioms.

\begin{definition}
We say that $A \in \C$ is of class $Del_{F_0}$ if it satisfies
\begin{enumerate}
    \item[(AX 1)] $j^0A \simeq F_0$ 
    \item[(AX 2)] $j^{k+1}A \simeq \tau_{p(k)} \iota_k j^k A$ $\forall k \in \{0, \dots, n \}$
\end{enumerate}
\end{definition}
Observe that in the second axiom $k=n$ means $A \simeq \tau_{p(n)} \iota_n j^n A$. We will abuse notation and say that $A \in Del_{F_0}$ (or simply $A \in Del$ if $F_0$ is clear from the context) whenever $A$ is of class $Del_{F_0}$

Now let us call 
$$ \textit{P}_{F_0}= \tau_{p(n)} \iota_n \dots \tau_{p(0)} \iota_0 F_0$$ 
We have the following theorem.

\begin{theorem} \label{teo deligne}
$A \in Del_{F_0} \Leftrightarrow A \simeq \textsl{P}_{F_0}$
\end{theorem}

So the axioms characterize in fact an isomorphism class in $\C$.

\begin{proof}
First suppose that $A$ satisfies the axioms, then
$$ A \simeq \tau_{p(n)}\iota_n j^n A \simeq \tau_{p(n)}\iota_n \tau_{p(n-1)}\iota_{n-1} j^{n-1} A \simeq \dots \simeq \tau_{p(n)}\iota_n \dots \tau_{p(0)}\iota_0 j^0 A $$
and $j^0 A \simeq F_0$, so we conclude that $A \simeq \textit{P}_{F_0}$

For the other direction we define
$$\textit{P}_{k}= \tau_{p(k)}\iota_{k} \dots \tau_{p(0)} \iota_0 F_0 $$
And we call $\textit{P}_{-1}=F_0$. Observe that
\begin{itemize}
    \item $\textit{P}_{k+1}= \tau_{p(k+1)}\iota_{k+1} \textit{P}_k$ and $\textit{P}_n = \textit{P}_{F_0}$
    \item $j_k \textit{P}_k \simeq \textit{P}_{k-1} $ for $k \in \{0, \dots, n \}$
    \item $j^k \textit{P}_{F_0} \simeq \textit{P}_{k-1}$ for $k \in \{0, \dots, n \}$
\end{itemize}
The first statement is direct from the definition of $\textit{P}_k$ and the third statement follows directly from the second. 

For the second statement we separate cases
\begin{itemize}
    \item For $k \geq 1$ we have $j_k \textit{P}_k = j_k \tau_{p(k)}\iota_{k} \textit{P}_{k-1} \simeq  \tau_{p(k)} j_k \iota_{k} \textit{P}_{k-1} \simeq \tau_{p(k)} \textit{P}_{k-1} \simeq \tau_{p(k)}\tau_{p(k-1)}\iota_{k-1} \textit{P}_{k-2} \simeq \tau_{p(k-1)}\iota_{k-1} \textit{P}_{k-2}=\textit{P}_{k-1}$
    \item For $k=0$ we have $j_0 \textit{P}_0= j_0 \tau_{p(0)} \iota_0 F_0 \simeq  \tau_{p(0)} j_0 \iota_0 F_0 \simeq \tau_{p(0)} F_0 \simeq F_0= \textit{P}_{-1}$
\end{itemize}
For $k=0$ the statement $j^k \textit{P}_{F_0} \simeq \textit{P}_{k-1}$ is the first axiom, and the second axioms is satisfied as $j^{k+1}\textit{P}_{F_0} \simeq \textit{P}_k= \tau_{p(k)}\iota_{k} \textit{P}_{k-1} \simeq \tau_{p(k)}\iota_{k} j^k \textit{P}_{F_0} $
\end{proof}
As the reader can imagine, there are many situations in which the conditions for stating the Deligne axioms are satisfied. The tricky element here is to find an object in $\C$ that would satisfy them.

To get the classical, topological, Deligne axioms we take for a chain of inclusions $ U_0 \hookar{\iota_0} U_1 \hookar{\iota_1} \dots \hookar{\iota_n} U_n=X  $ the derived categories $\C_k=D(U_k)$. We take $\tau_{p(k)}$ as truncations and for the functors $ (\iota_k, j_k)$  we take the derived functors of $(\iota_k)^\star$ and $(\iota_k)_\star$. 

Notice that with respect to what the statement of Deligne refers, we would not need to be in the derived categories, however we would loose track of the homological structure, which would make the axioms very insipid. We need to be wise in selecting the categories to consider.

Given the theory we have constructed, we can state the Deligne axioms and have theorem \ref{teo deligne} for any chain of simplicial complexes $ X_0 \hookar{\iota_0} X_1 \hookar{\iota_1} \dots \hookar{\iota_{n}} X_{n+1}=X $. Just like for the topological case, we take the categories $D_\D(X_k)$, the truncation as the $\tau_{p(k)}$'s, and as the morphisms we take the derivative of $(\iota_k)^\star$ and $(\iota_k)_\star$. 

In this setting, we have freedom of movement on the considered chain $X_0 \leq X_1 \leq \dots \leq X_{n+1}$, as well as in the topology on $\Sub(X)$. Furthermore, we could (if we wanted to) change the category from simplicial complexes into simplicial sets, delta sets, or even other presheaf categories. We will stay for now on simplicial complexes.


\subsection{Our setting: A PL stratified pseudomanifold} \label{section 6.2}

We now move into the result we obtained regarding the simplicial Deligne axioms. Perhaps the more puzzling of the freedoms just mentioned in last section is to get is the chain $X_0 \leq \dots \leq X$, since for topological spaces this chain is taken to be the \textit{open sets} $U_k=X-X_{n-k}$, which are obtained by a complement of sets, and as mentioned in chapter \ref{section simplicial complexes}, the complement of simplicial structures is not simplicial and our best replacement for complement, which we have denoted as $-_\D$, has a bad behaviour (when not subdivided). In particular, it does not help when working with a stratified space: If we consider for example the cone of $\partial \D^2$ with the cone point being $X_0$. If we consider $X-_\D X_0$ we get $\partial \D^2$, which is not homeomorphic to $|X|-|X_0|$.

Here is where subdivision comes into rescue. If we subdivide once we recover the three properties mentioned in section \ref{section simplicial complexes}. However we do not recover the topological structure as the example in the following picture shows

\begin{center}
\includegraphics[scale=0.2]{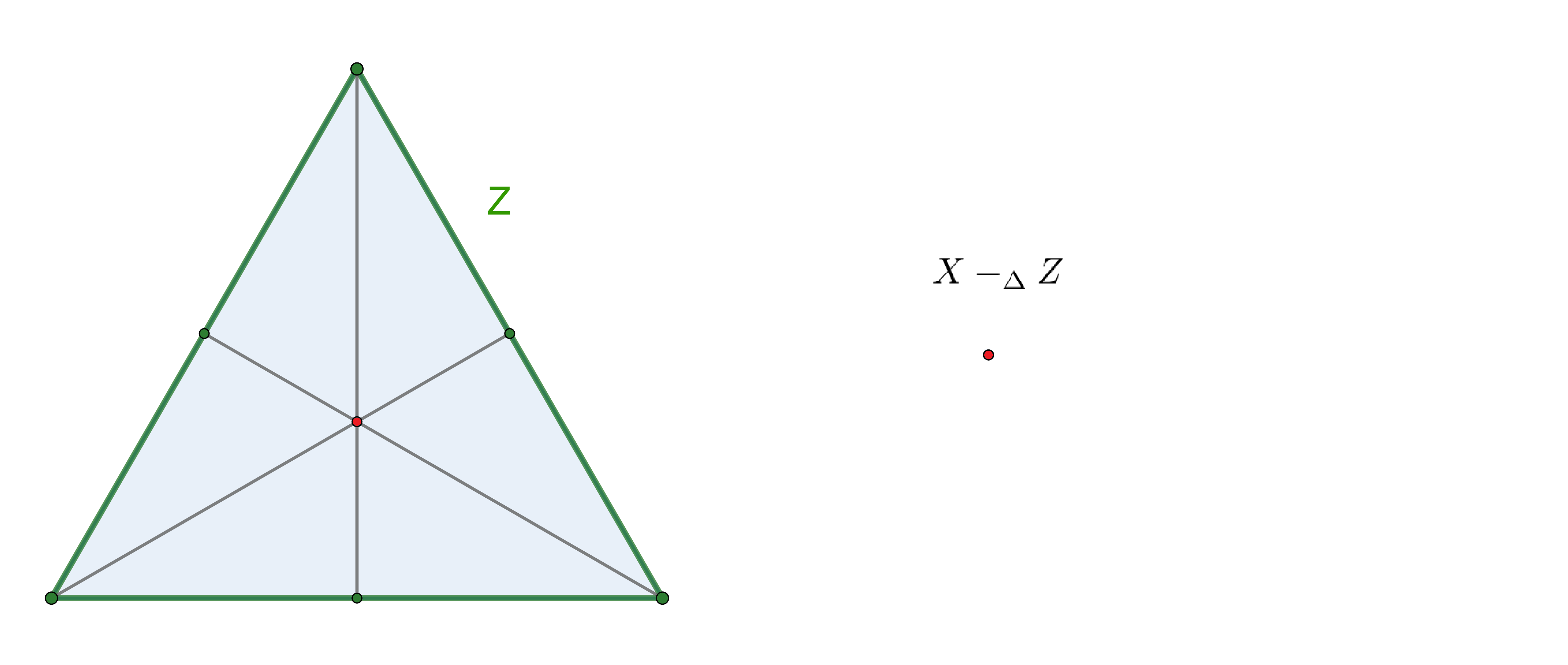}
\end{center}

Taking out $Z$ on that picture would leave us with just a point, which has empty interior. This is not what we want, and we subdivide again. The following picture give us a feeling that the topology is respected when taking out $Z$ when we take double subdivision

\begin{center}
\includegraphics[scale=0.2]{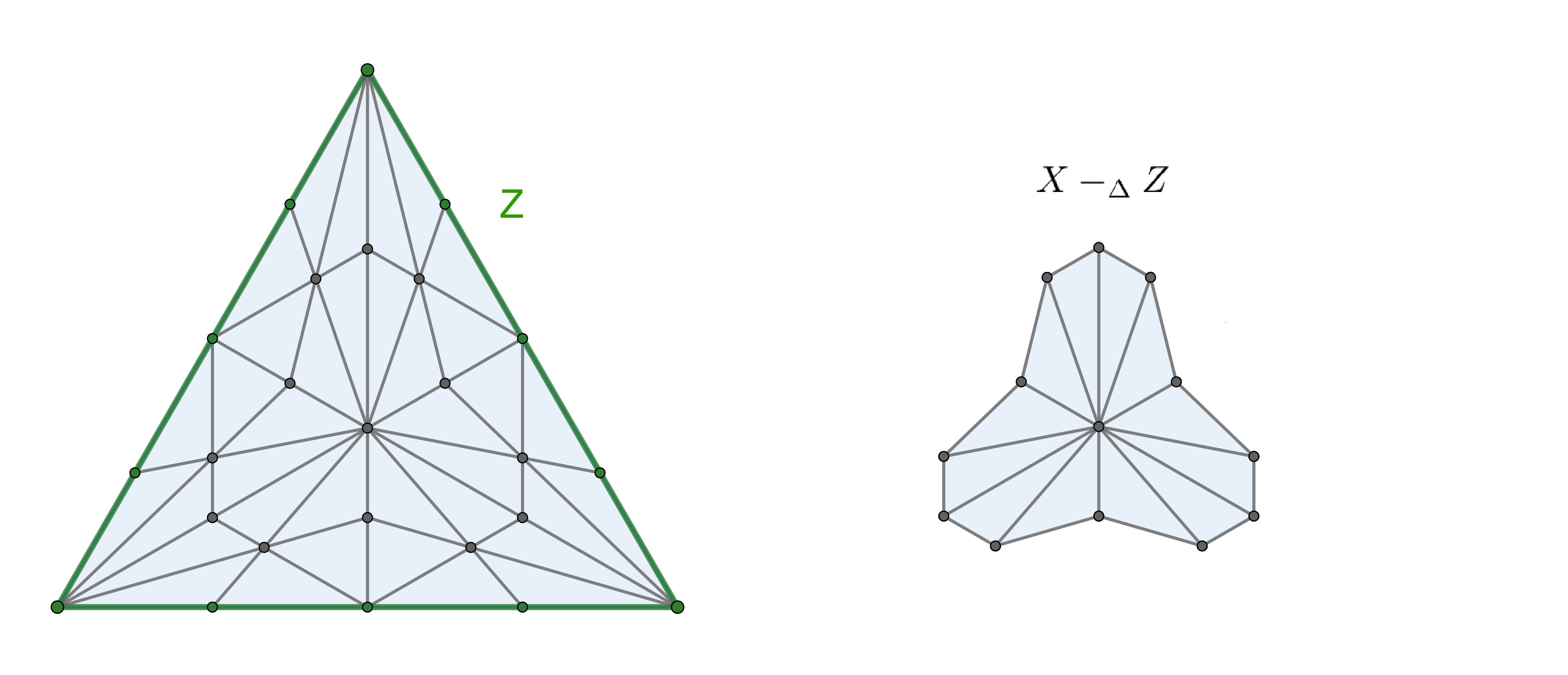}
\end{center}

On this document we are going to consider a simplicial complex that has been subdivided twice, particularly a simplicial complex that is the double subdivision of a fixed triangulation of a PL stratified pseudomanifold.

We will define now and make the mathematics for all we are loosely saying.

\subsubsection{$\D$-topologies on $Sd^2(X)$}

We start by considering $X'$ a simplicial complex such that $|X'|$ is a PL stratified pseudomanifold with stratification given by $|X_0| \subset |X_1| \subset \dots \subset |X_n|=|X'|$, and we subdivide it twice, so we consider $X= Sd^2(X')$ (and we call by abuse of notation $X_k=Sd^2(X_k)$).

We first take a look in the topology to consider. We have that the topology given by $\Sub(X)$ is not very good for our purposes, and this is because if we want to make Deligne happen, we need the smallest open sets of the topology to be somewhat distinguished neighborhoods, and $\Sub(X)$ is generated by $\ims$'s, which are far from being it.

There is one (or probably many, but we will focus on this one) topology that fulfils the task. To define it we first make a couple of definitions.

\begin{definition}
Given $X$ as before and $\s \in X'$, we define
\begin{itemize}
    \item $\lst(\s) = \st(b_\s)$ the star of the barycenter
    \item $\fst(\s) = \underset{\s < \tau}{\bigcup} \lst(\tau)$
\end{itemize}
\end{definition}

We present the following picture to clarify these definitions.

\begin{center}
    \includegraphics[scale=0.3]{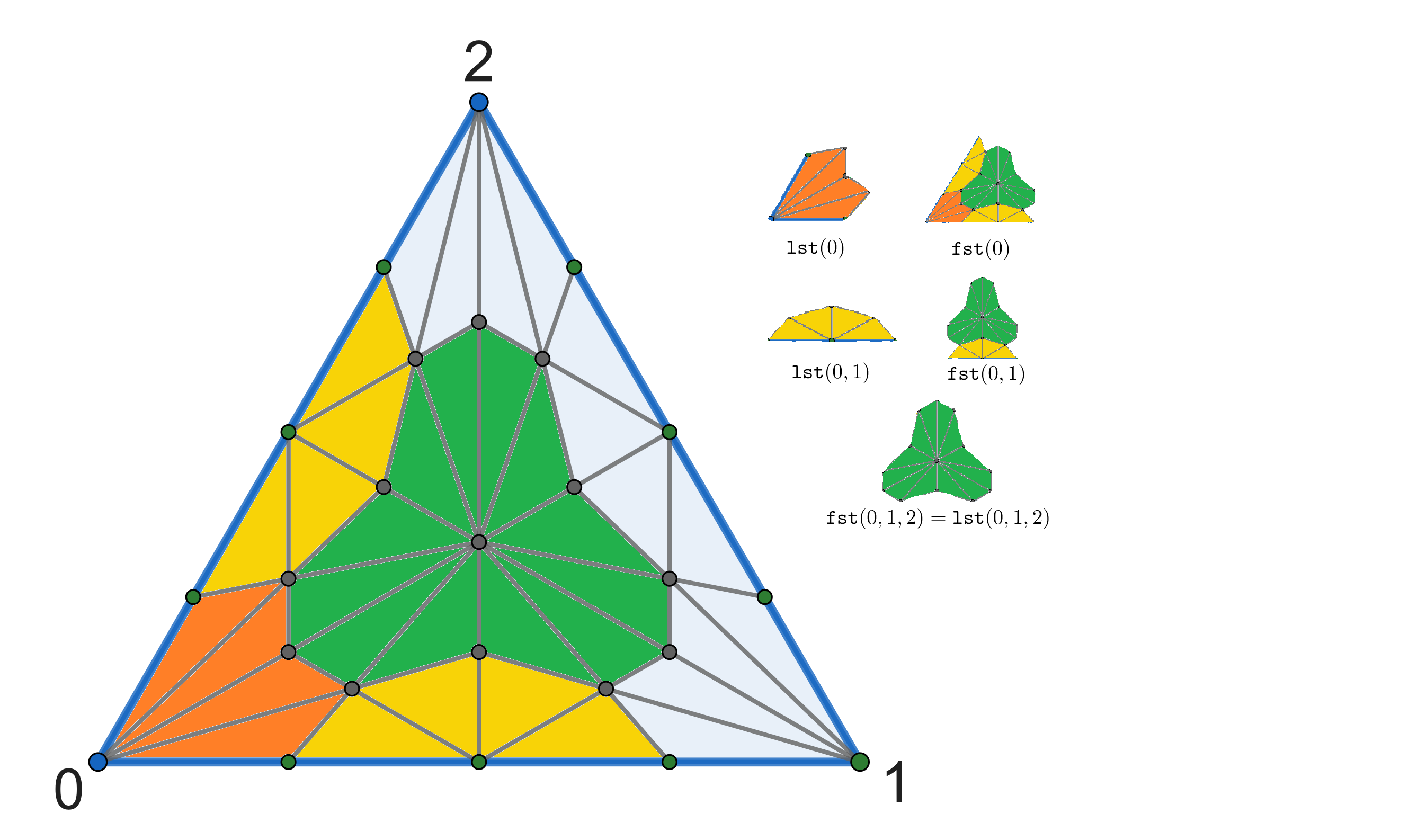}
\end{center}

And we take the topology whose base is $\{ \fst(\s) | \s \in X' \}$ and call it $$\mathcal{T}= < \{ \fst(\s) | \s \in X' \} >$$
This topology is very nice, since the base is closed under intersection, and this leads the stalks to be always evaluation on $\fst(\s)$'s.

\begin{lemma}\label{lemma stalk evaluation on fst}
Given $F \in Sh_\D^\mathcal{T}(X)$ and $\s \in X$
$$ F_\s= F(\fst(\s')) $$
Where $\s' \in X'$ is the lowest dimensional simplex with $|\ims| \subset |\Ima(\s')|$.
\end{lemma}

\begin{proof}
This is simply the fact that

$$F_\s= \colim{\s \in Y \in \mathcal{T}} F(Y) $$
Given $\s' \in X'$ is the lowest dimensional simplex with $|\ims| \subset |\Ima(\s')|$, $\fst(\s')= \underset{\s \in Y \in \mathcal{T}}{\bigcap}Y$ and since $\fst(\s') \in \mathcal{T}$ then $\colim{\s \in Y \in \mathcal{T}} F(Y)=F(\fst(\s'))$
\end{proof}

And now we need to establish a chain of subcomplexes to apply the Deligne axioms, we consider
$$ U_{k}^\D = X -_\D X_{n-k} $$
And we define $U_{k}^\D \hookar{\iota_k} U_{k+1}^\D $ the inclusions. We have that these subcomplexes are part of the topology, as the following proposition shows.
\begin{proposition} \label{proposition Uk in T}
For all $k \in \{0, \dots , n \}$
$$U_k^\D = \underset{b_\s \in |X|-|X_{n-k}|}{\bigcup}\fst(\s)$$
In particular $U_k^\D \in \mathcal{T}$
\end{proposition}
\begin{proof}
Consider first 
$$\eta= \{ \{ \eta_0^0 ,\dots, \eta_{l_0}^0 \} , \dots, \{ \eta_{0}^m ,\dots, \eta_{l_m}^m \} \} \in X-_\D X_{n-k}$$
this means that $\Ima(\eta) \cap Sd^2(X_{n-k}')= \emptyset$, which implies in particular that
$$ \{ \eta_{0}^j ,\dots, \eta_{l_j}^j \} \notin Sd(X_{n-k}') \text{ for all } j \in \{ 0,\dots, m \} $$
In particular for $j=0$, there is $\eta_i^0 \notin X_{n-k}'$. For this simplex $b_{\eta_i^0} \in |X|-|X_{n-k}|$, and we have that $\eta \in \lst(\eta_i^0) \leq \fst(\eta_i^0)$ (recall here that $ \{ \eta_{0}^j ,\dots, \eta_{l_j}^j \} \subset  \{ \eta_{0}^{j+1} ,\dots, \eta_{l_{j+1}}^{j+1} \}$ for all $j$)

On the other direction, consider
$$\eta= \{ \{ \eta_0^0 ,\dots, \eta_{l_0}^0 \} , \dots, \{ \eta_{0}^m ,\dots, \eta_{l_m}^m \} \} \in \lst(\tau)$$
with $\s, \tau \in X'$ such that $b_\s \in |X|-|X_{n-k}|$ and $\s < \tau$.

$\eta \in \lst(\tau)$ implies that $\tau=\eta_i^0$ for some $i$. Now, since $\s < \eta_i^0$ and $b_\s \in |X|-|X_{n-k}| $ we have that $\eta_i^0 \notin X_{n-k}$ which implies that $ \{ \eta_{0}^j ,\dots, \eta_{l_j}^j \} \notin Sd(X_{n-k}')$ for all $j$, which implies that $\Ima(\eta) \cap X_{n-k}=\emptyset $

\end{proof}

This implies that $(\iota_k)_\star F(Z)=F(U_{k}^\D \cap Z)$. Having this at hand we can state the Deligne axioms as shown in section \ref{abstract deligne axioms}. We will make the actual statement. Consider $D_\D(\mathcal{T})$ the derived category for $\mathcal{T}$, with $\tau_k$ the truncation functor and let $\bar{p}$ be a GM-perversity.

\begin{definition}
$F \in D_\D(\mathcal{T})$ is said to satisfy the $\D$-Deligne axioms if
\begin{enumerate}
    \item[(AX 1)] $ F|_{U_{0}^\D} \simeq \mathbb{R}$ 
    \item[(AX 2)] $F|_{U_{k+1}^\D} \simeq \tau_{\bar{p}(k)} R(\iota_k)_\star F|_{U_{k}^\D}$ $\forall k \in \{0, \dots, n \}$
\end{enumerate}
For a chain of functors satisfying these axioms, we say $F \in Del_\D$.
\end{definition}

And as shown in theorem \ref{abstract deligne axioms}, we have that the chains of functors that satisfy $\D$-Deligne are given up to quasi-isomorphism by $\tau_{\bar{p}(n)} R(\iota_n)_\star \dots \tau_{\bar{p}(2)} R(\iota_2)_\star C_\D $ 

We are now after sheaves that would satisfy the Deligne axioms, and with this in mind we develop the following functor.

\subsubsection{A good functor $\Phi: Sh(|X|) \to Sh_\D(X)$}\label{section functor Phi}

We define a functor
$$ \Phi : Sh(|X|) \to Sh_\D^\mathcal{T}(X) $$
Given by $\Phi(F)(Y)= \Gamma(|Y|,F)=\texttt{colim}_{|Y| \subseteq U}F(U)$. This is actually comes from the precomposition with the realization functor. We will abuse notation and call $\Gamma(|Y|,F)=F(|Y|)$.

We need to prove that $\Phi(F)$ is a sheaf, which we state in a proposition together with another useful fact

\begin{proposition} \label{proposition Phi soft to flasque}
Given the setting as before we have
\begin{enumerate}
    \item For all sheaves $F \in \Sh(|X|)$, we have that $\Phi(F) \in \Sh_\D^\T(X)$.
    \item If $F$ is soft then $\Phi(F)$ is flasque.
\end{enumerate}
\end{proposition}

\begin{proof}
We work presenting $\Gamma(|Y|,F)$ as $\{ (s,U) : |Y| \subseteq U,s \in F(U) \}/\sim$, where $(s,U) \sim (s',U')$ if there is an open set $V$ in $|X|$ with $|Y| \subseteq V \subseteq U \cap U'$ such that $s|_V=s'|_V$.

Now consider $Y=\bigcup\limits_{i \in I}Y_i$

For the first gluing condition, we take $(s,U),(t,U') \in \Gamma(|Y|,F)$ with $(s,U)|_{Y_i}=(t,U')|_{Y_i}$, that is, for each $i$ there exists $V_i$ with $|Y_i| \subseteq V_i \subseteq U \cap U' $ and $s|_{V_i}=t|_{V_i}$. Consider $V= \bigcup\limits_{i \in I}V_i$ which satisfies $|Y| \subseteq V \subseteq U \cap U'$, and $s|_V=t|_V$ by the gluing property of $F$, which is to say that $(s,U) \sim (t,U')$ in $\Gamma(|Y|,F)$.

For the second gluing condition, take $(s_i,U_i) \in \Gamma(|Y_i|,F) $ with $(s_i,U_i)|_{Y_i \cap Y_j} = (s_j,U_j)|_{Y_i \cap Y_j} $ for each $i$ and $j$, which means that there are open sets $V_{ij}$ with $|Y_i \cap Y_j| \subseteq V_{ij} \subseteq U_i \cap U_j$ such that $s_i|_{V_{ij}}=s_j|_{V_{ij}}$.

By topological properties we take open sets $V_i$ with $|Y_i| \subseteq V_i \subseteq U_i$ and $V_i \cap V_j \subseteq V_{ij}$. To see this, first give a metric $d$ to the (realization of the) simplicial complex. We may assume that each $Y_i$ belongs to the base, which implies that they are compact and that each $Y_i$ intersects only a finite amount of other $Y_j$'s. For each $i,j \in I$ consider $d(|Y_i|,U_i^c)=\epsilon_i$ and $d(|Y_i \cap Y_j|,V_{ij}^c)=\delta_{ij}$, and then
$$ V_i= \{ x \in |X| \text{ : } d(x,|Y_i|)< \texttt{min} \{ \epsilon_i, \epsilon_j, \delta_{ij}||Y_i \cap Y_j| \neq \emptyset \} \} $$

Now, by the gluing property of $F$ we obtain $s \in F(\bigcup\limits_{i \in I}V_i)$ such that $s|_{V_i}=s_i|_{V_i}$. Since $|Y| \subseteq \bigcup\limits_{i \in I}V_i$, this says that $(s, \bigcup\limits_{i \in I}V_i)|_{Y_i}=(s_i,U_i)$.

The second statement follows from the fact that (realization of) subcomplexes are closed subsets of $|X|$.
\end{proof}
Some remarks of this proposition are in order.

\vspace{0.2 cm}
\textbf{Remarks:}
\begin{itemize}
    \item The functor $\Phi$ has a left adjoint at the level of presheaves given by $\Psi(F)(U)=F(\bigcup\limits_{|Y|\subseteq U}Y)$ which sadly does not rise to the level of sheaves. However, a sheafification of $\Psi$ might open an interesting line of work.
    
    \item The second statement of the proposition tells us that by considering a simplicial model $X \to |X|$, we go up a level, since flasque is a stronger condition that soft. This means that in general \textit{soft sheaves are easier to find than flasque sheaves}. To illustrate this we present a result for which we will not develop the hole theory but it can be instinctively understood.
    \begin{center}
        If $\mathcal{R}$ is a \textit{soft} sheaf of algebras, any sheaf of $\mathcal{R}$-modules is \textit{soft}.
    \end{center}
    That is, softness is a hereditary condition. In (\cite{BlownupDel} Proposition 2.6), the authors use this result to prove that $\textbf{N}_{\bar{p}}^\star$ is soft, by using that it is a sheaf of $\textbf{N}_{\bar{0}}^\star$-modules (using the product $\textbf{N}_{\bar{0}}^\star \otimes \textbf{N}_{\bar{p}}^\star \too{- \cup -} \textbf{N}_{\bar{p}}^\star$), and that $\textbf{N}_{\bar{0}}^\star$ is soft, which is rather easy to prove.
    
    \item Since $\IpC$ is a chain of soft sheaves, we have that $\Phi(\IpC)$ is flasque. This in particular means that
    $$ \mathbb{H}^\star(X, \Phi(\IpC)) \simeq H^\star(\Phi(\IpC)(X)) \simeq H^\star(\IpC(|X|)) \simeq IH_\star^{\bar{p}}(|X|)$$
    That is, the hyperhomology of $\Phi(\IpC)$ on $X$ corresponds to the intersection homology of $|X|$. Compare this with the example given on section \ref{section simplicial sheaves}. We will develop this further in corollary \ref{main corollary}.
\end{itemize}

\subsection{Our main results} \label{section main result}

We have arrived now to the definitive moment, which is to find meaningful sheaves that satisfy the $\D$-Deligne axioms. There are in fact plenty of sheaves that satisfy the $\D$-Deligne axioms, since our good functor $\Phi$ acts as a machine of producing $\D$-Deligne sheaves.

During this section we will say "In the setting given on last section". Even though we just described this setting, given the importance of this results and in consideration of the lazy reader, we will recall the setting given on last section. 
\begin{itemize}
    \item We take $X'$ a simplicial complex subdivided twice with a chain of subcomplexes $X_0 \leq \dots \leq X_n = X'$ such that $|X'|$ is a PL stratified pseudomanifold with stratification $|X_0| \subset |X_1| \subset \dots \subset |X_n|=X$. We subdivide it twice $X=Sd^2(X')$ and call by abuse of notation $X_k=Sd^2(X_k)$.
    \item We consider the topology $\T \subset \Sub(X)$ with base $\{\fst(\s) | \s \in X' \}$. We take $U_k^\D = X-_\D X_k \in \T$ and $U_k^\D \hookar{\iota_k} U_{k+1}^\D$.
    \item We call $D_\D(\T)$ the derived category of $\Sh_\D^\T(X)$. On this category we say that $F \in Del_\D$ if
\begin{enumerate}
    \item[(AX 1)] $ F|_{U_{0}^\D} \simeq \mathbb{R}$ 
    \item[(AX 2)] $F|_{U_{k+1}^\D} \simeq \tau_{p(k)} R(\iota_k)_\star F|_{U_{k}^\D}$ $\forall k \in \{0, \dots, n \}$
\end{enumerate}
Theorem \ref{teo deligne} says that $F \in Del_\D \Leftrightarrow F \simeq \tau_{\bar{p}(n)} R(\iota_n)_\star \dots \tau_{\bar{p}(2)} R(\iota_2)_\star C_\D $ 
    \item We have a functor $\Phi: \Sh(|X|) \to \Sh_\D^\T(X)$ with formula $\Phi(F)(Y)= \Gamma(|Y|,F)= \colim{|Y| \subset U}F(U)$, which takes soft sheaves into flasque ones.
\end{itemize}

As said, $\Phi$ is a machine of producing $\D$-Deligne sheaves. One important example of this is that it transforms the sheaf $\IpC$ into a $\D$-Deligne sheaf.

\begin{proposition} \label{prop Phi(IpC) satisfies DDeligne}
In the setting given on last section, 
$$\Phi(\IpC) \in Del_\D $$
\end{proposition}

This proposition is not easy to prove, and we will spend the last section of the chapter proving it. Observe that the sheaf $\Phi(\IpC)$ is very \textit{meaningful} since, as remarked at the end of last section, its hyperhomology corresponds to intersection homology.
\begin{equation}\label{HH of Phi(IpC) is IH}
    \mathbb{H}^\star(X, \Phi(\IpC)) \overset{(1)}{\simeq} H^\star(\Phi(\IpC)(X)) \overset{(2)}{\simeq} H^\star(\IpC(|X|)) \overset{(3)}{\simeq} IH_\star^{\bar{p}}(|X|)
\end{equation}

This equivalence is easy to see, however we are going to remark the importance that it does have by explaining each isomorphism.
\begin{enumerate}
    \item[(1)] comes from the fact that $\Phi(\IpC)$ is flasque, so 
    $$ \mathbb{H}^i(X,\Phi(\IpC)= R^i \Gamma(X,\Phi(\IpC))= H^i(\Gamma(X,\Phi(\IpC)))=H^i(\Phi(\IpC(X)))$$
    \item[(2)] is because $\Phi(\IpC(X)))=\Gamma(|X|,\IpC)=\IpC(|X|)$.
    \item[(3)] Corresponds to the definition of (Borel-Moore) intersection homology.
\end{enumerate}

The proposition just presented has two important corollaries, that we state as theorems. The first of the corollaries tells us that the characterization up to quasi isomorphism of the Deligne axioms expresed in theorem \ref{teo deligne}, implies that not only $\Phi(\IpC)$ satisfies the $\D$-Deligne axioms, but also any other chain of sheaves in its quasi isomorphism class does.

\begin{theorem} \label{main theorem}
In the setting given in last section we have
$$F \in Del \Longrightarrow \Phi(F) \in Del_\D$$
\end{theorem}
\begin{proof}
Using theorem \ref{teo deligne}, if we have $F \in Del \Rightarrow F \simeq \IpC \Rightarrow \Phi(F) \simeq \Phi(\IpC)$, then by the last proposition, $\Phi(\IpC) \in Del_\D$, and then $\Phi(F) \in Del_\D$ (once again using theorem \ref{teo deligne}).
\end{proof}
The second corollary expands on equation (\ref{HH of Phi(IpC) is IH}). Putting together proposition \ref{prop Phi(IpC) satisfies DDeligne}, equation (\ref{HH of Phi(IpC) is IH}) and theorem \ref{teo deligne} we obtain a powerful corollary.

\begin{theorem} \label{main corollary}
In the setting given in last section, we have that if $F \in Del_\D$ then
$$ \mathbb{H}^\star(X,F) \simeq IH^p_\star(|X|) $$
\end{theorem}
So the $\D$-Deligne axioms provide a characterization of chains of sheaves, up to quasi isomorphism, that compute the classical intersection homology. This last theorem can be understood as a procedure that reads as follows
\begin{enumerate}
    \item Consider $X$ a PL stratified pseudomanifold
    \item Take $|K| \to X$ a compatible triangulation (such that each strata is a subcomplex), and subdivide it twice.
    \item If $F \in Del_\D$ then
    $$ \mathbb{H}^\star(K,F) \simeq IH^{\bar{p}}_\star(X) $$
\end{enumerate}

\subsection{Proof of the main result} \label{section proof of main result}

We will advocate us into the proof of the proposition now. The first $\D$-Deligne axiom is obvious. As for the second one, it is far from being so, and it will take us the rest of the document to see it.

\subsubsection{Simplifying (AX 2)}

We are going to prove the second axiom using the stalks
\begin{equation*} 
    (\Phi(\IpC)|_{U_{k+1}^\D})_\s \simeq (\tau_{p(k)} (\iota_k)_\star \Phi(\IpC)|_{U_{k}^\D})_\s \text{ for all } \s \in X
\end{equation*}
And we can use the fact that the subsimplicial complexes $U_k^\D$ are open to simplify this a little
\begin{equation} \label{equation stalks}
    (\Phi(\IpC))_\s \simeq (\tau_{p(k)} (\iota_k)_\star \Phi(\IpC)|_{U_{k}^\D})_\s \text{ for all } \s \in U_{k+1}^\D
\end{equation}

We recall from lemma \ref{lemma stalk evaluation on fst} that in $\mathcal{T}$ stalks are simple evaluations.
$$ F_\s= F(\fst(\s')) $$
So now (\ref{equation stalks}) becomes
\begin{equation*}
        \Phi(\IpC)(\fst(\s')) \simeq (\tau_{p(k)} (\iota_k)_\star \Phi(\IpC)|_{U_{k}^\D})(\fst(\s')) \text{ for all } \s \in U_{k+1}^\D
\end{equation*}
Now, by the observation following \ref{proposition Uk in T}, we have that for the right side of the equation
$$(\tau_{p(k)} (\iota_k)_\star \Phi(\IpC)|_{U_{k}^\D})(\fst(\s')) \simeq \tau_{p(k)} \IpC (|U_{k}^\D \cap \fst(\s')|) $$

So we are left to prove that

$$ \IpC (|\fst(\s')|) \simeq \tau_{p(k)} \IpC (|U_{k}^\D \cap \fst(\s')|) $$
For all $k \in \{0, \dots , n \}$ and $\s \in U_{k}^\D$. And this will be satisfied if the following lemma is true

\begin{lemma} \label{lemma 6}
Given $\s: \D^k \to X'$ with $b_\s \in |X_m|-|X_{m-1}|$, then we have that there exists a PL stratified pseudomanifold $L \subset |X|$ such that

\begin{enumerate}
    \item $|\fst(\s)| \simeq |\lst(\s)| \simeq \mathbb{R}^m \times cL $
    \item $| \fst(\s) -_\D X_{m} | \simeq \mathbb{R}^m \times L $
\end{enumerate}
(Where the isomorphism $\simeq$ here corresponds to stratified homotopy equivalence)
\end{lemma}
Then the results follows from the cone formula.

\subsubsection{$\fst(\s)$ is a distinguish neighborhood}

We advocate this section to prove lemma \ref{lemma 6}. So during this section, we will suppose that $X'$ is a triangulation of a PL-stratified pseudomanifold with stratification $X_0' \leq \dots \leq X_n'$. We take $X=Sd^2(X')$ and call $X_k= Sd^2(X_k')$. Consider then $\s: \D^k \to X'$ with its barycenter $b_\s= \{ \{ \s \} \} \in |X_m|-|X_{m-1}|$. 

We use abuse of notation and call for a simplex $\eta \in X'$ the subcomplex $Sd^2(\Ima(\eta))= \eta$.

The proof of lemma \ref{lemma 6} will be divided in five parts

\begin{enumerate}
    \item $\lst(\s) \simeq (\lst(\s) \cap \s) \times cL^\s $ for some $L^\s$.
    \item $|\lst(\s) \cap \s| \simeq B^k$ corresponds to a ball of dimension $k$.
    \item $|\fst(\s)| \simeq | \lst(\s) | $
    \item $| \lst(\s) | \simeq B^m \times cL $
    \item $| \lst(\s) -_\D X_m | \simeq B^m \times L$
\end{enumerate}

\vspace{0.1 cm}

\textbf{1. $\lst(\s) \simeq (\lst(\s) \cap \s) \times L^\s $}

\vspace{0.2 cm} 

We start by noticing that for a simplex $\eta \in X'$, given that (the double subdivision of) $\Ima(\eta)$ is fat, then
$$ \st_{\Ima(\eta)} (b_\eta)= \eta \cap \lst(\eta)  $$
We will call $\st_{\Ima(\eta)}(b_\eta)=\st_\eta(b_\eta)$, let us call

\begin{itemize}
    \item $S(\eta)= \{ A \in Sd(X') | \eta \in A \}$ (that is, vertices of the shape $\{ \s_0, \dots, \s_n, \eta, \tau_0, \dots, \tau_l \}$)
    \item $S_-(\eta)=\{ A \in S(\eta) | \eta= \max(A) \} $ (that is, vertices of the shape $\{ \s_0, \dots, \s_n, \eta \}$)
    \item For $B \in S_-(\eta)$ we call $$S_+(B) = \{A \in S(\eta)| B \subset A \text{ and } \eta < \tau \text{ } \forall \eta \in B \text{ } \forall \tau \in A-B \}$$ 
    Again this will consist on vertices of the shape $\{ \s_0, \dots, \s_n, \eta, \tau_0, \dots, \tau_l \}$, but now with $\{ \s_0, \dots, \s_n, \eta \}$ fixed. We call $S_+(\eta)= S_+(\{ \eta \})$.
\end{itemize}

Observe that $S(\eta)= \{ A \cup B | A \in S_-(\eta) \text{ } B \in S_+(\eta) \}$.

We have that $\lst(\eta)$ and $\st_\eta(b_\eta)$ take their shape according to this sets of vertices. Recall that $\G$ was defined in \ref{Def G}.

\begin{lemma}
For a simplex $\eta \in X'$ we have that
\begin{itemize}
    \item $\lst(\eta)= \G(S(\eta))$
    \item $\st_\eta(b_\eta)= \G(S_-(\eta))$
\end{itemize}
\end{lemma}
\begin{proof}
We show first that 
$$V(\st(b_\eta))=S(\eta)$$
Given $A \in S(\eta)$, we have that $A \in \{ \{\eta \} , A \}$ which is a $1$-simplex of $\st(b_\eta)$, so $A \in V(\st(b_\eta))$. And given $\fset{\s} \in \st(b_\s)$ we have that there is $\tau \in X$ such that $\fset{\eta} \in \tau$ and $\{ \eta \} \in \tau$. Since $\{ \eta \}$ is minimal, we must have that $\{\eta \} \subset \fset{\eta}$.

Now we proove that $\lst(\eta)$ is fat. Consider $\fset{\Psi} \in X$ (so each $\Psi_k$ has a shape $\fset{\eta}$ and $\Psi_k \subset \Psi_{k+1}$ for all $k$) with $\Psi_k \in V(\lst(\eta))$ for all $k$, so $\Psi_0 \in S(\eta)$, in particular $\eta \in \Psi_0$, and then for $B= \{ \{ \eta \}, \Psi_0,\dots, \Psi_n \}$ we have that $\fset{\Psi} \subset B$ and $\{ \eta \} \in B$ and therefore $\fset{\Psi} \in \st(b_\eta)$

For the second statement, first consider that $\st_\eta(b_\eta)= \st(\eta) \cap Sd^2(\Ima(\tau)) = \G(S(\eta)) \cap \G(Sd(\Ima(\tau)))= \G(S(\eta) \cap Sd(\Ima(\tau)))$ by \ref{prop G}, so it will be enough to prove that
$$ S(\eta) \cap Sd(\Ima(\tau)) =S_-(\eta) $$
which comes from the fact that for all $\fset{\eta} \in Sd(\Ima(\tau))$ we have that $\eta_k \subset \eta$ for all $k$.
\end{proof}

Now, given $A= \{\s_0, \dots, \s_m, \s \} \in S_-(\s)$ we define
\begin{itemize}
    \item $L^\s = \G(S_+(\s)-\{ \{ \s \} \})$
    \item $L^A = \G( S_+(A)- \{ A \} ) $
\end{itemize}
We have the following lemma.
\begin{lemma}
In the situation above
\begin{itemize}
    \item $L^\s \simeq L^A$ for all $A \in S_-(\s)$
    \item $\G(S_+(\s)) = cL^\s$ and $\G(S_+(A)) = cL^A$
\end{itemize}
\end{lemma}

\begin{proof}
For the first statement we use a morphism that in vertices will take $\{\s , \tau_0, \dots, \tau_l \} \mapsto \{\s_0, \dots, \s_m, \s , \tau_0, \dots, \tau_l \}$. It is easy to see that this defines a morphism and furthermore an isomorphism with inverse given in vertices by $ \{\s_0, \dots, \s_m, \s , \tau_0, \dots, \tau_l \} \mapsto \{\s , \tau_0, \dots, \tau_l \}$.

For the cone shape, observe that every simplex of $\G(S_+(A))$ has the shape $\fset{\Psi}$ with $A \subset \Psi_0$, so it is contained in $\{A \} \cup \{ \Psi_0, \dots, \Psi_n  \}$. The case for $cL^\s$ is analogous.
\end{proof}

So we want to prove that
$$\G(S(\s)) \simeq \G(S_-(\s)) \times \G(S_+(\s)) = \G(S_-(\s) \times S_+(\s)) $$
(With the second equality coming from \ref{prop G}),
and this is true because the bijection
$$ f: S_-(\s) \times S_+(\s) \to S(\s) $$
defined by $(A,B) \mapsto A \cup B$ defines an isomorphism
$$ f: \G(S_-(\s) \times S_+(\s)) \to \G(S(\s)) $$
We use \ref{prop G2} to prove this
\begin{lemma}\label{lemma 625}
In the situation above, the function $ f: S_-(\s) \times S_+(\s) \to S(\s) $ defines an isomorphism
$$ f: \G(S_-(\s) \times S_+(\s)) \to \G(S(\s)) $$
\end{lemma}

\begin{proof}
We need to prove that $f$ and its inverse, say $g$, define morphisms, that is, they send simplices to simplices.

So first consider a simplex in $\G(S_-(\s) \times S_+(\s))$. This simplex will have the shape

$$ \{ (\{ \s_{0}^0, \dots, \s_{k_0}^0, \s \} , \{ \s, \tau_{0}^0, \dots, \tau_{l_0}^0 \} ), \dots , (\{ \s_{0}^m, \dots, \s_{k_m}^m, \s \} , \{ \s, \tau_{0}^m, \dots, \tau_{l_m}^m \} ) \}  $$
With the simplex $\{ \{ \s_{0}^0, \dots, \s_{k_0}^0, \s \}, \dots , \{ \s_{0}^m, \dots, \s_{k_m}^m, \s \} \}$ being in $\G(S_-(\s))$, and $\{  \{ \s, \tau_{0}^0, \dots, \tau_{l_0}^0 \} , \dots , \{ \s, \tau_{0}^m, \dots, \tau_{l_m}^m \} \}$ being in $\G( S_+(\s))$, that is 
$$ \{ \s_{0}^{i}, \dots, \s_{k_{i}}^{i}, \s \} \subset \{ \s_{0}^{i+1}, \dots, \s_{k_{i+1}}^{i+1}, \s \} $$
and 
$$\{ \s, \tau_{0}^{i}, \dots, \tau_{l_{i}}^{i} \} \subset \{ \s, \tau_{0}^{i+1}, \dots, \tau_{l_{i+1}}^{i+1} \}$$
for all $i$. When applying $f$ to this simplex we obtain 
$$ \{ \{ \s_{0}^0, \dots, \s_{k_0}^0, \s , \tau_{0}^0, \dots, \tau_{l_0}^0 \}, \dots , \{ \s_{0}^m, \dots, \s_{k_m}^m, \s, \tau_{0}^m, \dots, \tau_{l_m}^m \} \}  $$
Given the two inclusions before we obtain that 
$$\{ \s_{0}^i, \dots, \s_{k_i}^i, \s , \tau_{0}^i, \dots, \tau_{l_i}^i \} \subset \{ \s_{0}^{i+1}, \dots, \s_{k_{i+1}}^{i+1}, \s , \tau_{0}^{i+1}, \dots, \tau_{l_{i+1}}^{i+1} \}$$ 
for all $i$, and therefore it is a simplex of $\G(S(\s))$

On the opposite direction, consider a simplex on $\G(S(\s))$
$$ \{ \{ \s_{0}^0, \dots, \s_{k_0}^0, \s , \tau_{0}^0, \dots, \tau_{l_0}^0 \}, \dots , \{ \s_{0}^m, \dots, \s_{k_m}^m, \s, \tau_{0}^m, \dots, \tau_{l_m}^m \} \} $$
Applying $g$ will give us

$$ \{ (\{ \s_{0}^0, \dots, \s_{k_0}^0, \s \} , \{ \s, \tau_{0}^0, \dots, \tau_{l_0}^0 \} ), \dots , (\{ \s_{0}^m, \dots, \s_{k_m}^m, \s \} , \{ \s, \tau_{0}^m, \dots, \tau_{l_m}^m \} ) \}  $$

Now, since $\{ \s_{0}^i, \dots, \s_{k_i}^i, \s , \tau_{0}^i, \dots, \tau_{l_i}^i \} \subset \{ \s_{0}^{i+1}, \dots, \s_{k_{i+1}}^{i+1}, \s, \tau_{0}^{i+1}, \dots, \tau_{l_{i+1}}^{i+1} \}$ for all $i$, we will have that
$$ \{ \s_{0}^{i}, \dots, \s_{k_{i}}^{i}, \s \} \subset \{ \s_{0}^{i+1}, \dots, \s_{k_{i+1}}^{i+1}, \s \} $$
and 
$$\{ \s, \tau_{0}^{i}, \dots, \tau_{l_{i}}^{i} \} \subset \{ \s, \tau_{0}^{i+1}, \dots, \tau_{l_{i+1}}^{i+1} \}$$
for all $i$, giving us that the image by $g$ is a simplex of $\G(S_-(\s) \times S_+(\s))$
\end{proof}

\vspace{0.1 cm}

\textbf{2. $|\lst(\s) \cap \s| \simeq B^k$}

\vspace{0.2 cm}
We have that $\lst(\s) \cap \ims \subset \overset{\circ}{\s}$, which is a manifold. Since
$ \lst(\s) \cap \ims = \st_{\ims}(b_\s)$, we get that it is a ball by (\cite{RourkeSanderson} Corollary 2.20).

\vspace{0.3 cm}

\textbf{3. $|\fst(\s)| \simeq |\lst(\s)|$}

\vspace{0.2 cm}
First observe that what we have proved so far for $\s$ is also true for any simplex of $X'$.

We are going to inductively (on dimension) define a retraction from $\fst(\s)= \underset{\s < \tau}{\bigcup} \lst(\tau)$ onto $\lst(\s)$ y reducing each $\lst(\tau)$ for $\s < \tau$. Before starting with the retraction we make a couple of observations

\begin{lemma} \label{lemma 621}
In the setting above, we have
\begin{enumerate}
    \item For $\eta \in X'$ with $b_{\eta} \in |X_j|-|X_{j-1}|$ we have that $\lst(\eta) \cap \eta \leq X_j-_\D X_{j-1}$
    \item For $\tau, \tau' \in X'$ we have that $\lst(\tau) \cap \lst(\tau') \neq \emptyset \Leftrightarrow \tau < \tau' \text{ or } \tau' < \tau$
    \item If we have $\tau, \tau' \in \beta(\s)= \{ \tau \in X' | \s < \tau \}$ of the same dimension with $\tau \neq \tau'$ then $\lst(\tau) \cap \lst(\tau') = \emptyset$
\end{enumerate}
\end{lemma}
\begin{proof}
For the first statement, consider $ \{ \Psi_0, \dots , \Psi_k \} \in \st_\eta(b_\eta) $, that is, $\Psi_i \subset \Psi_{i+1}$ for all $i$, and they have the shape 
$$\Psi_i= \{ \s_0^i, \dots, \s_{l_i}^i, \eta \}$$
with $\s_0^i \subset \dots \subset \s_{l_i}^i \subset \eta$. We want to show that $\{ \Psi_0, \dots , \Psi_k \} \in X_j-_\D X_{j-1}$, which is equivalent to prove that
$$\{ \Psi_0, \dots , \Psi_k \} \subset Sd(X_j')- Sd(X_{j-1}') $$
So we want that $\{ \s_0^i, \dots, \s_{l_i}^i, \eta \} \in Sd(X_j')- Sd(X_{j-1}')$ for all $i$. And for this it will be enough to prove that $\eta \in X_j-X_{j-1}$. Observe that $b_\eta \in |X_j|-|X_{j-1}|$ means that $ \{ \eta \} \in Sd(X_j')- Sd(X_{j-1}')$ which is equivalent to what we want.

For the second statement observe that $\lst(\tau) \cap \lst(\tau') = \G(S(\tau) \cap S(\tau'))$, so
$$\lst(\tau) \cap \lst(\tau')= \emptyset \Leftrightarrow S(\tau) \cap S(\tau') = \emptyset$$
Now, if we have $ A \in S(\tau) \cap S(\tau')$, we would have that both $\tau, \tau' \in A$, which means that $\tau < \tau'$ or $\tau' < \tau$. 
On the other direction, if $\tau \subset \tau'$ then we have $\{\tau, \tau' \} \in S(\tau) \cap S(\tau')$, and if $\tau' \subset \tau$ then $\{\tau', \tau \} \in S(\tau) \cap S(\tau')$.

The third statement is a direct consequence of the second, since $\tau, \tau'$ of the same dimension with $\tau \neq \tau'$ means that neither is face of the other.
\end{proof}

We need one more lemma before proceeding, we call for $\tau \in X'$ with $\s < \tau$
$$ \fst_\tau(\s)= \underset{\s \subset \eta \subsetneq \tau}{\bigcup}\lst(\eta)$$

\begin{lemma} \label{lemma 622}
In the setting before, for $\tau \in X'$ with $\s < X'$ we have that
\begin{itemize}
    \item $\fst_\tau(\s) \cap \lst(\tau) \leq \partial \lst(\tau)$.
    \item $|\fst_\tau(\s) \cap \lst(\tau) \cap \tau|$ is simply connected.
\end{itemize}
\end{lemma}
\begin{proof}
First note that we can state everything in terms of fat subcomplexes, since
\begin{itemize}
    \item $\lst(\tau)=\G(S(\tau))$
    \item $\partial \lst(\tau)=\G(S(\tau)- \{\{ \tau \}\} )$
    \item $\lst(\eta)= \G(S(\eta))$
\end{itemize}
Now, $\lst(\tau) \cap \underset{\s \subset \eta \subsetneq \tau}{\bigcup}\lst(\eta)= \G(S(\tau)) \cap \underset{\s \subset \eta \subsetneq \tau}{\bigcup}\G(S(\eta)) \leq \G(S(\tau)) \cap \G(\underset{\s \subset \eta \subsetneq \tau}{\bigcup}S(\eta))=\G(S(\tau) \cap \underset{\s \subset \eta \subsetneq \tau}{\bigcup}S(\eta))$, so to prove the first statement, it will suffice to show that
$$S(\tau) \cap \underset{\s \subset \eta \subsetneq \tau}{\bigcup}S(\eta) \subset S(\tau)- \{\{ \tau \}\} $$
And this is true since $\{ \tau \} \notin \underset{\s \subset \eta \subsetneq \tau}{\bigcup}S(\eta)$ (because the only set of size one of $S(\eta)$ is $\{ \eta \}$, and all $\eta$'s in the union are different from $\tau$)

Now we prove that in fact 
$$\underset{\s \subset \eta \subsetneq \tau}{\bigcup}\G(S(\eta)) = \G(\underset{\s \subset \eta \subsetneq \tau}{\bigcup}S(\eta)) $$
To see this consider $ \fset{\Psi} \in \G(\underset{\s \subset \eta \subsetneq \tau}{\bigcup}S(\eta))$, that is $\Psi_i \in \underset{\s \subset \eta \subsetneq \tau}{\bigcup}S(\eta)$ for all $i$.

So for all $i$ there is $\eta_i$ such that $\Psi_i \in S(\eta_i)$, that is $\eta_i \in \Psi_i$. Notice that since $\Psi_0 \subset \dots \subset \Psi_n$, we have that $\eta_0 \in \Psi_i $ for all $i$, which means that $\Psi_i \in S(\eta_0)$ for all $i$, so $\fset{\Psi} \in \G(S(\eta_0)) \subset \underset{\s \subset \eta \subsetneq \tau}{\bigcup}\G(S(\eta))$.

This means that $\fst_\tau(\s) \cap \lst(\tau) \cap \tau$ is a fat subcomplex of the simply conected simplicial complex $Sd^2(\Ima(\tau))$, which gives us that $|\fst_\tau(\s) \cap \lst(\tau) \cap \tau|$ is simply conected.
\end{proof}

We are now ready to make our proof

\begin{proposition}
In the setting given above, $|\fst(\s)|$ is homotopically stratified equivalent to $|\lst(\s)|$.
\end{proposition}
\begin{proof}
We will contract each $\lst(\tau)$ with $\s < \tau$, inductively on the dimension of $\tau$ starting on dimension $n$ down to dimension $k+1$ (recall that $\s$ has dimension $k$)

First, for a $n$-simplex $\tau$, we have that $\lst(\tau)= \st_\tau(b_\tau)$ is an $n$-ball contained in $|X_n|-|X_{n-1}|$, so we can contract it to a point in a way that respects the stratification. In particular it is homotopically stratified equivalent to $\lst(\tau) \cap \fst_\tau(\s)$ (which is simply connected by lemma \ref{lemma 622}).

By lemma \ref{lemma 621}, we can apply the corresponding homotopy independently on each simplex of dimension $n$ containing $\s$.

Now, suppose we have reduced every simplex containing $\s$ of dimension grater than $j$, and consider $\tau$ of dimension $j$. Recall that $\lst(\tau) \simeq (\lst(\tau) \cap \tau) \times cL^\tau$. Since $\lst(\tau)\cap \tau$ is a ball, we can contract it to a simply connected subset of its border, as for example $\lst(\tau) \cap \fst_\tau(\s) \cap \tau$ (by lemma \ref{lemma 622}). We name this homotopy $h_\tau$. Again by lemma \ref{lemma 621}, this homotopy respects stratification. Consider now
$$ h_\tau \times Id: |\lst(\tau) \cap \tau| \times |cL^\tau| \times [0,1] \to |\lst(\tau) \cap \tau| \times |cL^\tau| $$
This will give an homotopy that contracts $\lst(\tau)$ to its border. By lemma \ref{lemma 621}, we can apply this homotopy independently in simplices of dimension $j$, and it will respect the stratification since $h_\tau$ does.
\end{proof}

\vspace{0.1 cm}

\textbf{4. $| \lst(\s) | \simeq B^m \times cL $}

\vspace{0.2 cm}
Now consider for the point $b_\s \in |X_m|-|X_{m-1}|$, a distinguish neighborhood $N_\s= B^m \times cL$, with $L$ a PL-stratified pseudomanifold of dimension $n-m$ compatible with the stratification. We can consider $N_\s$ small enough so that
$$ |N_\s| \subset |\lst(\s)| $$
We consider the simplicial complex $\lst(\s) \cap N_\s$. We can do a pseudoradial projection, and then make it into a simplicial complex (see the techniques throughout section \cite{RourkeSanderson} Chapter 2), obtaining a subdivision $K'$ of $\lst(\s)$ in which $\st_{K'}(b_\s)=N_\s$, then by (\cite{RourkeSanderson} Lemma 2.19), we have that $lk(b_\s,\lst(\s))$ is homeomorphic to $lk(b_s,K')$, and now we have that
$$N_\s = \st_{K'}(b_\s)=b_\s lk(b_s,K') \simeq b_\s lk(b_\s,\lst(\s))=\lst(\s) $$
Observe that since the pseudoradial projection goes along the simplices, we have that this homeomorphism respects stratification.

\vspace{0.1 cm}

\textbf{5. $| \lst(\s) -_\D X_m | \simeq B^m \times L$}

\vspace{0.2 cm}
Observe that the same homotopy that contracts $\fst(\s)$ into $\lst(\s)$ will contract $\fst(\s)-_\D X_m$ into $\lst(\s)-_\D X_m$, and therefore if we prove that $| \lst(\s) -_\D X_m | \simeq B^m \times L$, then we will be done proving lemma \ref{lemma 6}.

We will prove that $| \lst(\s) -_\D X_m | \simeq B^m \times L$ by showing that
$$| \lst(\s) -_\D X_m | \simeq | \lst(\s)| - |X_m |  $$
and then using that the latest is equivalent to $B^m \times L$, as it is a distinguished neighborhood. We make this proof in inductive steps. 

Observe first that
$ \lst(\s) -_\D X_m = \lst(\s) -_\D (\underset{\tau \in X_m'}{\bigcup} \Ima(\tau)) = \underset{\tau \in X_m'}{\bigcap} (\lst(\s) -_\D \Ima(\tau)) = \underset{\s < \tau \in X_m'}{\bigcap} (\lst(\s) -_\D \Ima(\tau)) =\lst(\s) -_\D( \underset{\s < \tau \in X_m'}{\bigcup} \Ima(\tau)) $. Furthermore, we can suppose that all these $\tau$'s in the union are $m$ dimensional. Observe that for $X$ locally finite, there are finitely many elements in $\{ \tau \in X'| \s < \tau \in X_m' \}$. We will enumerate them
$$\{ \tau \in X'| \s < \tau \in X_m' \}= \{\tau_1, \dots , \tau_l \}$$
For later convenience, we take $\tau_l$ to be $\s$ itself. Let us define the following.
\begin{definition}
Given $\{\tau_1, \dots , \tau_l \}$ as before, we call
$$ \mst{k} = \lsts -_\D (\Ima(\tau_1) \cup \dots \cup \Ima(\tau_k)) $$
And we set $\mst{0}=\lsts$
\end{definition}
Observe that by the ninth statement of \ref{prop -D}, we have that $$\mst{k+1}= \mst{k} -_\D \Ima(\tau_{k+1})$$
And now we will prove that we can make the proof inductively. 

\begin{lemma}\label{lemma 628}
In the setting given before, if we have that for all $k \geq 0$
$$ |\mst{k} -_\D \Ima(\tau_{k+1})| \simeq |\mst{k}| - |\Ima(\tau_{k+1})|$$
Then $| \lst(\s) -_\D X_m | \simeq | \lst(\s)| - |X_m |  $
\end{lemma}
\begin{proof}
The proof goes inductively as follows
\begin{equation*}
\begin{split}
|\lsts -_\D X_m | & = |\mst{l}| \\
 & \simeq |\mst{l-1}|-|\Ima(\tau_{l})| \\
 & \simeq (|\mst{l-2}|-|\Ima(\tau_{l-1})|)-|\Ima(\tau_{l})|\\
 & = |\mst{l-2}|-(|\Ima(\tau_{l-1})| \cup |\Ima(\tau_{l})|)\\
 & \simeq \dots \simeq |\lsts| - \underset{\s < \tau \in X_m'}{\bigcup} |\Ima(\tau)|
\end{split}
\end{equation*}
And $|\lsts| - \underset{\s < \tau \in X_m'}{\bigcup} |\Ima(\tau)| = |\lsts|-|X_m|$. The only tricky step here is to see that 
$$|\mst{k+1}|- \underset{j \geq k}{\bigcup}|\Ima(\tau_j)| \simeq (|\mst{k}|-|\Ima(\tau_{k+1})|)- \underset{j \geq k}{\bigcup}|\Ima(\tau_j)| $$
The homotopy equivalence here is constructed by restricting the corresponding homotopy of $|\mst{k} -_\D \Ima(\tau_{k+1})| \simeq |\mst{k}| - |\Ima(\tau_{k+1})|$. This can be done since the inclusions of subspaces
$$ |\mst{k+1}|- \underset{j \geq k}{\bigcup}|\Ima(\tau_j)| \hookrightarrow |\mst{k+1}| $$
Are homotopy equivalences. This is because subtracting $|\Ima(\tau_j)|$ corresponds to subtracting simplices of dimension $<n$ out of $\lsts$, hence we can retract $|\mst{k+1}|$ into $|\mst{k+1}|- \underset{j \geq k}{\bigcup}|\Ima(\tau_j)|$. And if the homotopy equivalence of $|\mst{k} -_\D \Ima(\tau_{k+1})| \simeq |\mst{k}| - |\Ima(\tau_{k+1})|$ respects the stratification, so will its restriction.
\end{proof}
We will treat $\s$ differently from the other simplices of $\{\tau_0, \dots , \tau_l \}$. We will also assume that each of these other simplices has dimension $m$. We will make use of the following lemma

\begin{lemma}
In the setting given before
$$\lst(\s) -_\D \ims \simeq (\lsts \cap \s) \times L^\s$$
\end{lemma}
\begin{proof}
Recall by lemma \ref{lemma 625} that
$$\lsts \simeq \G(S_-(\s)\times S_+(\s))$$
via the identification
$$\{\s_0, \dots, \s_m, \s, \tau_0, \dots, \tau_l \} \mapsto (\{\s_0, \dots, \s_m, \s \}, \{ \s, \tau_0, \dots, \tau_l \}  )$$
When removing $\ims$, we are subtracting all vertices of the shape $\{\s_0, \dots, \s_m, \s \}$, that is, all $(\{\s_0, \dots, \s_m, \s \}, \{ \s \})$, and then
$$\lsts -_\D \ims \simeq \G(S_-(\s) \times (S_+(\s)-\{ \{ \s \} \})) = B_\s \times L^\s$$
\end{proof}
And then we have

\begin{lemma} \label{lemma 6210}
In the setting given before
$$|\lsts|-|\ims| \simeq |\lst(\s) -_\D \ims| $$
\end{lemma}
\begin{proof}
The proof goes as follows
\begin{equation*}
\begin{split}
|\lsts|-|\ims| & = |B_\s \times cL^\s|-|\ims| \\
 & = (|B_\s| \times |cL^\s|)-|\ims| \\
 & = |B_\s| \times (0,1) \times |L^\s| \\
 & \simeq |B_\s| \times |L^\s| = |B_\s \times L^\s| = |\lst(\s) -_\D \ims|
\end{split}
\end{equation*}
\end{proof}

We will now subtract the rest of the simplices of $X_m$, but before start doing it we will take a closer look on $L^\s$, suppose first that we are working in a building block, that is, $X= \Dn$

\begin{lemma}
Suppose that $X= \Dn$ and $\s= (0, \dots, k)$, then
$$L^\s \simeq Sd^2(\D^{n-k-1})$$
Where $\D^{n-k-1}$ is formed by the vertices of the simplex $(k+1, \dots, k+(n-k))$
\end{lemma}
\begin{proof}
Recall that $L^\s= \G(S(\s)-\{\{\s\}\})$

We establish a bijection between $S(\s)-\{\{\s\}\}$ and $Sd(Im((k+1, \dots,n)))$. To a face $(a_{j_0}, \dots, a_{j_m})<(k+1, \dots,n)$ we assign 
$$f((a_{j_0}, \dots, a_{j_m}))=(0, \dots, k, a_{j_0}, \dots, a_{j_m})$$
Which naturally defines a function
$$f: Sd(Im((k+1, \dots,n))) \to S(\s)-\{\{\s\}\} $$
Given by $f(\{ \eta_0, \dots, \eta_j \})= \{ f(\eta_0), \dots, f(\eta_j) \}$. This is easily seen to be a function and define a morphism
$$f: \G(Sd(Im((k+1, \dots,n)))) \to \G(S(\s)-\{\{\s\}\}) $$
Furthermore, it has an inverse defined in the same fashion $$g((0, \dots, k, a_{j_0}, \dots, a_{j_m}))=(a_{j_0}, \dots, a_{j_m})$$
\end{proof}

And then, in the general case, $L^\s$ is the union of these pieces.

\begin{lemma}
In the setting given before, let $\{ \eta_i \}_{i \in I}$ be the set of $n$-simplices of $X$, then
$$ L^\s= \underset{i \in I}{\bigcup}L^\s_{\eta_i}$$
where $L^\s_{\eta_i}$ is the construction of $L^\s$ assign to the subspace $|\Ima(\eta_i)|$.
\end{lemma}
\begin{proof}
The proof is simply the fact that
$$L^\s= \underset{i \in I}{\bigcup}L^\s \cap \eta_i= \underset{i \in I}{\bigcup} \G(\{A \in S_+(\s)|max(A)<\eta_i \}- \{\{\s\}\})$$
\end{proof}

Now observe that in the double subdivision, when taking off faces of some $\D^m$ with either $-$ or $-_\D$, the homotopy type does not change

\begin{lemma}
Consider $Z \leq \partial \Dm$ a simplicial complex formed by proper faces of an $m$-simplex. Then
$$|\Dm| = |Sd^2(\Dm)| \simeq |Sd^2(\Dm) -_\D Z | \simeq |Sd^2(\Dm)| - |Z|$$
Furthermore, if $X$ is a union of $m$-simplices and $Z \leq \partial X$, then
$$|X| = |Sd^2(X)| \simeq |Sd^2(X) -_\D Z | \simeq |Sd^2(X)| - |Z|$$
\end{lemma}
\begin{proof}
We use that in general, for $\eta \in \Dm$ we have
$$|\lst(\eta) -_\D \Ima(\eta)| \simeq |\lst(\eta)| - |\Ima(\eta)| $$
which is proven in the same fashion as lemma \ref{lemma 6210}
\begin{equation*}
\begin{split}
|\lst(\eta)|-|\Ima(\eta)| & = |B_\eta \times cL^\eta|-|\Ima(\eta)| \\
 & = |B_\eta| \times (0,1) \times |L^\eta| \\
 & \simeq |B_\eta| \times |L^\eta| = |B_\eta \times L^\eta| = |\lst(\eta) -_\D \Ima(\eta)|
\end{split}
\end{equation*}
We use this homotopy inductively on all simplices of $Z$, starting from dimension $0$ up to dimension $\dim(Z)$, and in this way we obtain
$$|Sd^2(\Dm)-_\D Z | \simeq |Sd^2(\Dm)| - |Z| $$
The remainder equivalence $|Sd^2(\Dm)| \simeq |Sd^2(\Dm)| - |Z|$ is obvious since we have $|Z| \subset \partial |\Dm| $.

The same process will work for $X$ being union of $m$-simplices, although it might not be obvious that we can iterate the homotopies to get to $|X|-|Z|$, for this consider the following
\begin{equation*}
\begin{split}
|Sd^2(X)-_\D Z| & = | \underset{\tau \in X-Z}{\bigcup}\lst(\tau) \cup \underset{\tau \in Z}{\bigcup}\lst(\tau) -_\D \underset{\tau \in Z}{\bigcup}\Ima(\tau)| \\
 & = \underset{\tau \in X-Z}{\bigcup}|\lst(\tau)| \cup \underset{\tau \in Z}{\bigcup}|\lst(\tau) -_\D \Ima(\tau)| \\
 & \simeq  \underset{\tau \in X-Z}{\bigcup}|\lst(\tau)| \cup \underset{\tau \in Z}{\bigcup}|\lst(\tau)| - |\Ima(\tau)| = |X|-|Z|
\end{split}
\end{equation*}
\end{proof}

We can now prove the last part of the result.

\begin{lemma}
In the setting given before, we have that for all $k \geq 0$
$$ |\mst{k} -_\D \Ima(\tau_{k+1})| \simeq |\mst{k}| - |\Ima(\tau_{k+1})|$$
\end{lemma}

Recalling lemma \ref{lemma 628}, this will give us the final result.
\begin{proof}
Recall by lemma \ref{lemma 625} that
$$\lsts \simeq \G(S_-(\s)\times S_+(\s))$$
via the identification
$$\{\s_0, \dots, \s_m, \s, \tau_0, \dots, \tau_l \} \mapsto (\{\s_0, \dots, \s_m, \s \}, \{ \s, \tau_0, \dots, \tau_l \}  )$$
Now, when we perform $-_\D \Ima(\tau)$ we are subtracting vertices of the form 
$$( \{ \s_0, \dots, \s \}, \{\s, \tau_0, \dots, \tau_l \})$$ With $\tau_l < \tau$, so we have that
\begin{equation*}
\begin{split}
\lst(\s)-_\D \Ima(\tau) &  = \G(S_-(\s)) \times \G(S_+(\s)-\{A \in S_+(\s)| max(A)<\tau \}) \\
 & \simeq B_\s \times (L^\s -_\D \G(\{A \in S_+(\s)| max(A)<\tau \}) )
\end{split}
\end{equation*}
Now, since $\dim(\tau)<n$ then $\G(\{A \in S_+(\s)| max(A)<\tau \}) \leq \partial L^\s$, and since $L^\s$ is a union of $m$-simplices, we can use the previous lemma to conclude (Observe here that taking off all the simplices finally will take a subcomplex contained in the frontier of $L^\s$, which does not change the homotopy type (which was the hole point of last lemma). We also take $Id \times -$ to create the homotopy).
\end{proof}

\newpage
\pagenumbering{roman}
\setcounter{page}{1}
\printbibliography

\end{document}